\documentclass[headsepline,paper=A4, 12pt, pagesize, english]{amsart}
\usepackage[a4paper,hscale=0.7,vscale=0.75,centering]{geometry}

\usepackage{amssymb,mathrsfs}
\usepackage{color}
\usepackage{amsmath} 
\usepackage{amsthm}
\usepackage[english]{babel}
\usepackage[T1]{fontenc}
\usepackage[utf8]{inputenc}
\usepackage{hyperref}
\usepackage{graphicx} 

\newcommand{\bs}[1]{\boldsymbol{#1}}

\newtheorem{thm}{Theorem}

\newtheorem{ass}{Assumption}
\newtheorem{cor}{Corollary}
\newtheorem{lem}{Lemma}

\theoremstyle{definition}

\newtheorem{example}{Example}
\newtheorem{rem}{Remark}

\newcommand{\sProd}[2]{\left\langle #1, #2 \right\rangle} 
\newcommand{\norm}[1]{\left\lvert #1 \right\rvert} 
\newcommand{\wDist}{\mathcal W} 
\newcommand{\quadV}[1]{\left[ #1 \right]}
\newcommand{\coV}[2]{\left[ #1, #2 \right]}

\newcommand{\reals}{\mathbb{R}}

\newcommand{\naturals}{\mathbb{N}}

\newcommand{\filt}{\mathcal{F}}

\newcommand{\init}{\mu}
\newcommand{\law}[1]{\operatorname{Law}(#1)}

\newcommand{\generator}{\mathcal{L}}
\renewcommand{\phi}{\varphi}

\title[Sticky couplings of diffusions]
{Sticky couplings of multidimensional diffusions with different drifts}

\date{\today}

\begin{document}

\author{Andreas Eberle}
\thanks{Financial support from the German Science foundation through the {\em Hausdorff Center for Mathematics} is gratefully acknowledged.}
\address{Universit\"at Bonn, Institut f\"ur Angewandte Mathematik,
  Endenicher Allee 60, 53115 Bonn, Germany}
\email{eberle@uni-bonn.de}
\urladdr{wt.iam.uni-bonn.de}

\author{Raphael Zimmer}
\thanks{}
\address{Universit\"at Bonn, Institut f\"ur Angewandte Mathematik,
  Endenicher Allee 60, 53115 Bonn, Germany}
\email{Raphael.Zimmer@uni-bonn.de}
\urladdr{wt.iam.uni-bonn.de}

\date{}

\begin{abstract}
We present a novel approach of coupling two 
multidimensional and non-degenerate It{\^o} processes $(X_t)$ and $(Y_t)$ which follow dynamics with different drifts.
Our coupling is \emph{sticky} in the sense that there is a stochastic process $(r_t)$, which solves
a one-dimensional stochastic differential equation with a \emph{sticky boundary} behavior at zero, such that
almost surely $\norm{X_t-Y_t}\leq r_t$ for all $t\geq 0$. 
The coupling is constructed as a
weak limit of Markovian couplings. We provide explicit, non-asymptotic and long-time stable bounds for the
probability of the event $\{X_t=Y_t\}$. 
\end{abstract}

\maketitle

\section{Introduction}
Let $(B_t)$ and $(\tilde{B}_t)$ be $d$-dimensional Brownian motions.
We consider two diffusion processes with values in $\reals^d$ which follow dynamics with different drifts, i.e.\ 
 \begin{alignat}{4}\label{eq:DiffusionDifferentDrifts}
 	dX_t \ &= \ b(t,X_t) \ dt \ &+ \  dB_t, &\qquad &X_0 \ &= \ &x,\\
 	dY_t \ & = \ \tilde{b}(t,Y_t)\ dt\  &+ \ d\tilde{B}_t, &\qquad &Y_0 \ &=\  &y. \label{eq:DiffusionDifferentDrifts2}
 \end{alignat}
We assume that the drift coefficients
 $b,\tilde{b}:\reals_+\times \reals^d\rightarrow\reals^d$ are locally Lipschitz. Moreover, we impose assumptions
 which imply that a geometric Lyapunov drift condition holds for \eqref{eq:DiffusionDifferentDrifts} and
 that there is a constant $M>0$ such that uniformly $|b-\tilde{b}|\leq M$.
 \par\smallskip
 Diffusions with different drifts occur in many application areas. For example, one could consider a Langevin diffusion $(X_t)$
 and a perturbation or approximation $(Y_t)$ of the latter. Other natural examples are McKean-Vlasov processes, where the drift coefficients
 depend not only on the current position of the process but also on the corresponding law. A natural question arising
 is how to obtain explicit bounds for the distance of $X_t$ and $Y_t$ in Kantorovich distances, e.g.\ in total variation norm.
 There are a few articles which try to answer this question in a general setting: Using Girsanov's theorem and coupling on the path space,
 the works \cite{MR814659,MR827946,MR1786115} establish bounds on the total variation norm of such diffusions.
 In \cite{MR3522009}
 bounds for the distance between transition probabilities of diffusions with different drifts are derived using analytic arguments,
 see also the related work \cite{2015arXiv150704014M}. The drawback of these approaches is that the derived bounds are typically only 
 useful for small time horizons and are not long-time stable.
 The article \cite{MR3343646} provides
 bounds for the distance between stationary measures of diffusions with different drifts. Coupling methods are used
 in \cite{2016arXiv160501559D} to provide long-time stable bounds on the distance between a Langevin diffusion and its Euler approximation.
 Howitt constructs in \cite{Howitt} a \emph{sticky coupling} of two one-dimensional Brownian motions with different drifts
 using time-change arguments which are restricted to the one-dimensional setting.
 \par\smallskip
 In this article, we discuss a novel approach of constructing couplings $(X_t,Y_t)$ of solutions to
 \eqref{eq:DiffusionDifferentDrifts} and  \eqref{eq:DiffusionDifferentDrifts2} in a multi-dimensional setting.
 Consider for example the case where $\tilde{b}$ differs from $b$ by a non-zero constant $m$, i.e., $\tilde{b}(t,x)=b(t,x)+m$ for
 some $m\in\reals^d$, and let $(X_t)$ and $(Y_t)$ be solutions of \eqref{eq:DiffusionDifferentDrifts} and \eqref{eq:DiffusionDifferentDrifts2}
 respectively. In this case, whenever $X_t$ and $Y_t$ meet, the drift forces the processes
 to immediately move apart from each other. It is clear that, regardless of how the processes are coupled, one cannot hope for the existence of an almost surely finite
 stopping time $T$ such that $P[ X_t=Y_t\  \forall t\geq T]=1$. Nevertheless, we construct a coupling such that for any
 given $t>0$, we have $P[X_t=Y_t]>0$ and the coupling is $\emph{sticky}$ in the sense that there is a 
 continuous 
 semimartingale $(r_t)$ which solves a one-dimensional stochastic differential equation with a \emph{sticky boundary} behavior at zero
 such that almost surely $\norm{X_t-Y_t}\leq r_t$ for all $t\geq 0$. This allows us to establish explicit, non-asymptotic and long-time stable bounds for the
probability of the event $\{X_t=Y_t\}$. The coupling is constructed as a weak limit of
 Markovian couplings. The idea for the coupling is based on \cite{Eberle2015,2016arXiv160606012E} where
 coupling approaches for particle systems and nonlinear McKean-Vlasov processes are discussed, cf.\ Section \ref{secMcKeanVlasov} for a comprehensive comparison. We show that sticky couplings can be applied
effectively to provide total variation bounds between the laws of both linear
and non-linear diffusions with varying drifts.
 \par\smallskip
 \textbf{Outline:} The main results are presented in Section \ref{secMain}. 
 In Section
\ref{sec:Sticky}  we recall results on the existence and uniqueness of one-dimensional SDEs with sticky boundary,
we establish an approximation result for the latter, and we study the long-time behavior of solutions to such equations using
coupling methods. Based on these results, the proof of our main theorem and the construction of the \emph{sticky coupling}
are presented in Section \ref{sec:proofs}.
\begin{figure}
 \centering
 \begin{minipage}[t]{0.9\textwidth}
 \begin{center}
 \includegraphics[width=9cm]{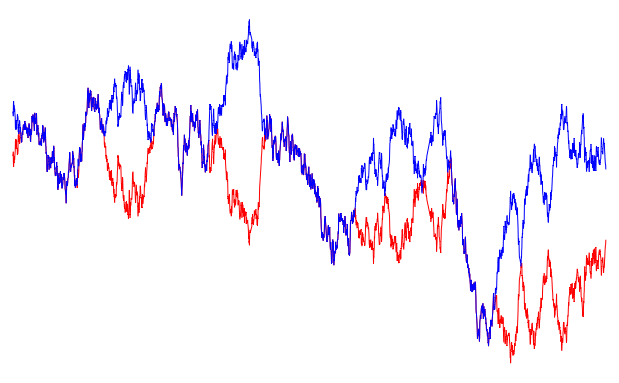} 
 \end{center}
 \caption{\label{fig1} A sticky
coupling of two diffusions on $\mathbb R^1$}
 \end{minipage}
 \end{figure}
\par\smallskip

\section{Main results}\label{secMain}
\subsection{Sticky couplings}
We impose the following assumptions: 
\begin{ass} \label{ass:BoundOnDriftDifference}
	There is a constant $M\in [0,\infty)$ such that 
	\begin{equation*}
		\norm{b(t,x)-\tilde{b}(t,x)}\leq M\qquad\mbox{for any }x\in\reals^d\mbox{ and }t\geq 0.
	\end{equation*}
\end{ass}
\begin{ass}\label{ass:CurvControlForB}
 	There is a Lipschitz function $\kappa:[0,\infty)\rightarrow \reals$ such that
 	\begin{equation*}
 		\sProd{x-y}{b(t,x)-b(t,y)} \leq \kappa(\norm{x-y}) \cdot \norm{x-y}^2 \quad\text{for any $x,y\in\reals^d$ and $t\geq 0$}.
 	\end{equation*}
 	Outside of a bounded interval, the function $\kappa$ is constant and strictly negative.
\end{ass}

The assumptions imply in particular that the unique strong 
solutions $(X_t)$ and $(Y_t)$ of \eqref{eq:DiffusionDifferentDrifts} and \eqref{eq:DiffusionDifferentDrifts2} respectively are non-explosive.
We present our main result:

 \begin{thm}[Sticky coupling]\label{main1}
Suppose that Assumptions \ref{ass:BoundOnDriftDifference} and \ref{ass:CurvControlForB} hold true. Then for any initial values $x,y\in\reals^d$,
 there is a coupling  $(X_t,Y_t)$ of solutions to \eqref{eq:DiffusionDifferentDrifts} and 
  \eqref{eq:DiffusionDifferentDrifts2}, respectively,
 such that $X_t-Y_t$ is \emph{sticky at zero} in the sense that the 
 difference is controlled by a solution of a one-dimensional SDE with a {sticky boundary behavior} at zero.
More precisely, there is a real-valued process $(r_t)$ solving the SDE
\begin{eqnarray}\label{stickyappl}
	dr_t &=& \left(M\ +\ \kappa(r_t)\,r_t\right) \ dt \ + \ 2\, I(r_t>0) \ dW_t, \qquad r_0\  = \ \norm{x-y},
\end{eqnarray}
driven by a one-dimensional Brownian motion $(W_t)$, such that
almost surely,
\begin{eqnarray}\label{monotonappl}
	 \norm{X_t-Y_t}\ \leq\  r_t \quad\mbox{ for any } t\geq 0.
\end{eqnarray}
The process $(r_t)$ is \emph{sticky} at zero in the sense that almost surely, 
\begin{eqnarray} 
	2M\,  \int_0^t I(r_s=0)\  ds  & = & \ell_t^0(r), \qquad 0\ \leq \ t \ < \ \infty,
\end{eqnarray}  
where $\ell_t^0(r)$ is the right local time at $0$ of $(r_t)$, i.e.,
\begin{eqnarray*}
\ell_t^0(r)\ =\ \lim_{\epsilon\downarrow 0}\frac{1}{\epsilon}\int_0^t I(0\le r_s <\epsilon )\, d[r]_s\ =\ 
4 \, \lim_{\epsilon\downarrow 0}\frac{1}{\epsilon}\int_0^t I(0<r_s<\epsilon )\, ds.
\end{eqnarray*}
Equation \eqref{stickyappl} admits an invariant probability
measure $\pi$. For $M=0$, $\pi =\delta_0$, and for $M>0$, $\pi $ is determined by \begin{eqnarray}\label{pimain}
  \pi(dx) &\propto& \left(\frac{2}{ M}\ \delta_0(dx) \ +\ \exp\left(\frac 12\int_0^x(M+\kappa(y)\,y)\, dy\right)\lambda_{(0,\infty)}(dx)\right).
 \end{eqnarray} 
 If the initial conditions coincide, i.e., if $x=y$, then for any $t\geq 0$,
  \begin{eqnarray}\label{t1}
 	\quad P[X_t=Y_t]&\geq & \pi[\{0\}] \ =\ \left( 1+\frac M2\int_0^\infty
 	\exp\left(\frac 12\int_0^x (M+\kappa(y)\,y)\, dy\right)\, dx\right)^{-1} .
 \end{eqnarray}
In general, there are constants $c,\epsilon\in(0,\infty)$, depending only on $M$ and $\kappa$, such that for any $t>0$ and any initial
values $x,y\in\reals^d$,
 \begin{eqnarray}
 	\label{tt2}	P[X_t\not=Y_t]&\leq& \frac{1}{\epsilon} \ \frac{c}{e^{c\,t}-1} \
 	\norm{x-y}\ + \ \pi[(0,\infty)].
 \end{eqnarray}
 The constants $c$ and $\epsilon$ are given by
 \begin{eqnarray*} 	
	c  = \left(2{\int_0^{R_1} \frac{\Phi(s)}{\phi(s)}\ ds}\right)^{-1} \quad\text{and}\quad \epsilon 
	=  \min\left\{\left(2\int_0^{R_1} \frac{1}{\phi(s)}\ ds\right)^{-1},\; c\,\Phi(R_1) \right\},
 \end{eqnarray*}
 where  $\phi(r)= \exp\left(-\frac{1}{2} \int_0^r \left(M\ +\ \kappa(s)\,s\right)^+ ds\right)$, $ \Phi(r) =  \int_0^r \phi(s) \, ds$,
 \begin{eqnarray}
 \label{R0main}		\quad R_0 &=& \inf\{ R\geq 0 : \left(M\ +\ \kappa(r)\,r\right)\leq 0\quad\text{for any } r\geq R \} , \quad\text{and}\\
 	\label{R1main}	\quad R_1 &=& \inf\{ R\geq R_0 : R(R-R_0)\;\left(M/r \  +\ \kappa(r)\right) \leq -4 \quad\text{for any } r\geq R  \}.
 \end{eqnarray}
  \end{thm} 
 
In Section \ref{sec:Sticky} we also provide explicit bounds
 on the expected values $E[\norm{X_t-Y_t}]$, cf. Theorem \ref{quantification1}
 further below.\smallskip
 
The coupling $(X_t,Y_t)$ in Theorem \ref{main1} is constructed as a weak limit of Markovian couplings. The construction of the coupling and the proof of the theorem are given in Section \ref{sec:proofs}. 
  
  \begin{rem}[Reflection coupling]
 The classical reflection coupling of Lindvall and Rogers \cite{MR841588} occurs as a special case of the coupling 
 in Theorem \ref{main1} when the drift coefficients
     coincide, i.e., $b=\tilde{b}$. In this case we can choose $M=0$
     so that $0$ is an absorbing  boundary for the diffusion
     process $(r_t)$. The equation \eqref{t2} reduces to 
   \begin{eqnarray}
 	\label{t2}   
     P[X_t\not=Y_t]&\leq& \frac{1}{\epsilon} \ \frac{c}{e^{c\,t}-1} \ \norm{x-y},
     \end{eqnarray}
    which is a well-known bound for reflection coupling \cite{MR841588, MR972776}.
  \end{rem}
   
In the two special cases $M=0$ and $x=y$, the bound in
\eqref{t2} takes a very simple and intuitive form. In general, however, the rate $c$ depends on $M$. This dependence can be
 avoided by considering a modified coupling.
 
 \begin{thm}\label{thm2}
 	There is a coupling $(\tilde{X}_t,\tilde{Y}_t)$ of solutions to
 	\eqref{eq:DiffusionDifferentDrifts} and
 	\eqref{eq:DiffusionDifferentDrifts2} such that 
 	\begin{eqnarray}
 		\label{TV2} P[\tilde{X}_t\not = \tilde{Y}_t] &\leq&
 		\frac{1}{\tilde{\epsilon}}\, \frac{\tilde{c}}{e^{\tilde{c}t}-1}\, \norm{x-y}
 		\,+\, \pi[(0,\infty)]\quad\mbox{ for any }t\ge 0,
 	\end{eqnarray}
 	where $\tilde{c},\tilde{\epsilon}$ are defined analogously to $c$ and $\epsilon$
 	but with $M=0$.
 \end{thm}
 
 \begin{proof}[Proof of Theorem \ref{thm2}]
 	Consider a process $(Z_t)$ satisfying
 	\begin{eqnarray*}
 		dZ_t &=& b(t,Z_t)\, dt \, + \, dB_t, \qquad Z_0=y.
 	\end{eqnarray*}
 	Let $(\tilde{X_t},\tilde{Z}_t)$ be a standard reflection coupling of $(X_t)$
 	and $(Z_t)$, i.e., a sticky coupling in the case where the drifts coincide.
 	Then we can glue this coupling with a sticky coupling of $(Z_t)$ and $(Y_t)$,
 	i.e., there are processes $(\tilde{X_t},\tilde{Z_t},\tilde{Y_t})$ defined
 	on a joint probability space such that $(\tilde{X}_t,\tilde{Z}_t)$ is a sticky coupling of
 	$(X_t,Z_t)$, and $(\tilde{Z}_t,\tilde{Y}_t)$ is a sticky coupling of
 	$(Z_t,Y_t)$, see e.g.\ the ``glueing lemma'' in \cite{MR2459454}. 	For $t\ge 0$, we obtain
 	by Theorem \ref{main1}:
$$
 		P[\tilde{X}_t\not=\tilde{Y}_t] \ \leq \ P[\tilde{X}_t\not = \tilde{Z}_t] \, +
 		\, P[\tilde{Z}_t\not = \tilde{Y}_t]
 		\ \leq \ \frac{1}{\tilde{\epsilon}}\, \frac{\tilde{c}}{e^{\tilde{c}t}-1}\, \norm{x-y}
 		\,+\, \pi[(0,\infty)].
$$
 \end{proof}

To make the bounds in the theorems more explicit, we now assume that we are given constants $\mathcal R,L\in [0,\infty )$ and $K\in (0,\infty )$ such that
for any $t\geq 0$,
 \begin{equation}\label{KLRbound}
 		\sProd{x-y}{b(t,x)-b(t,y)}\ \leq\ 
\begin{cases} 		
 		L \norm{x-y}^2 &\ \text{for any $x,y\in\reals^d$},\\
 		-K \norm{x-y}^2 &\ \text{for $x,y\in\reals^d$ s.t.\ $|x-y|\ge\mathcal R$}.
 		\end{cases}
 	\end{equation}
Hence Assumption \ref{ass:CurvControlForB} is satisfied with $\kappa (r)=L
\, I(r<\mathcal R)-K\, I(r\ge \mathcal R)$.
In this case, the exponential decay rate $\tilde c$
in Theorem \ref{thm2} is bounded from below by 
\begin{equation*}
\tilde c^{-1}\ \le\ \begin{cases}
4\,\max (\mathcal R^2,\, K^{-1}) &\mbox{if }L=0,\\
3e\, \max (\mathcal R^2,\, 4K^{-1})&\mbox{if }L\mathcal R^2\le 4,\\
8\sqrt\pi L^{-1/2}(L^{-1}+K^{-1})\mathcal R^{-1}\exp\left({L\mathcal R^2}/{4}\right)+16K^{-2}\mathcal R^{-2}&\mbox{if }L\mathcal R^2> 4,
\end{cases}
\end{equation*}
see Lemma 1 in \cite{Eberle2015} (Note that the definitions of the function
$\kappa$ and the constant $c$ in \cite{Eberle2015} differ from the definitions above by a factor $-2$, $2$, respectively). The following lemma provides explicit 
upper bounds on the long-time asymptotics of the probabilities in
\eqref{t2} and \eqref{TV2}. The proof is included in Section \ref{sec:proofs}.

\begin{lem}\label{lem:example} Suppose that Condition \eqref{KLRbound} is
satisfied.
Then $\pi [(0,\infty )]=\alpha /(1+\alpha )$ where $\alpha$ is a non-negative constant such that for $M\le K\mathcal R$,
$$
 \alpha \ \le \ \left(\pi^{1/2}e^{1/2}K^{-1/2}+2\mathcal R
\max (4,L\mathcal R^2+2M\mathcal R)^{-1}\right)\, M\, \exp \left(M\mathcal R/2+L\mathcal R^2/4 \right),$$
and for $M\ge K\mathcal R$,
\begin{eqnarray*}
	\alpha\leq
 		\left(\sqrt{\frac{\pi}{K}}+\frac{2\mathcal{R}}{\max(4,2M\mathcal{R}+L\mathcal
 		R^2)}\right)M \exp\left(\frac{M^2}{4K}+\frac{L+K}{4}\mathcal R^2\right).
\end{eqnarray*}
\end{lem}

The theorems imply bounds on the total variation distance between
the laws of $X_t$ and $Y_t$ for any time $t\ge 0$.  We now verify that in 
two simple examples, the
bound in \eqref{TV2} is of the correct order:

\begin{example}[Ornstein-Uhlenbeck processes]\label{exa:ou}
Fix $m\in\reals^d\setminus\{0\}$.  
We consider Orn\-stein-Uhlenbeck processes on $\reals^d$, given by
 \begin{alignat}{4}\label{eq:DiffusionDifferentDrifts3}
 	dX_t \ &= \ -X_t/2 \ dt \ &+ \  dB_t, &\qquad &X_0 \ &= \ &x,\\
 	dY_t \ & = \ -(Y_t-m)/2\ dt\  &+ \ d\tilde{B}_t, &\qquad &Y_0 \ &=\  &y, \label{eq:DiffusionDifferentDrifts4}
 \end{alignat}
 where $(B_t)$ and $(\tilde{B}_t)$ are $d$-dimensional Brownian motions. 
 Let $d(t)$ denote the total variation distance between the laws of $X_t$ and
 $Y_t$ at time $t$. It is well-known that
 $X_t$ and $Y_t$ are normally distributed with 
\begin{eqnarray*}
	\law{X_t} &=& \mathcal{N}\left(e^{-t/2}\ x ,\  (1-e^{-t}) \operatorname{I}_{d}
	\right),\\
	\law{Y_t} &=& \mathcal{N}\left(e^{-t/2}\ y \ + \ (1-e^{-t/2})\ m ,\  (1-e^{-t}) \operatorname{I}_{d} \right).
\end{eqnarray*}  
The total variation distance between $d$-dimensional normal distributions
$\mathcal N(a,b I_d)$ and $\mathcal N(\tilde{a},b I_d)$ with
$a,\tilde{a}\in\reals^d$ and $b\in(0,\infty)$ is given by 
$\Phi_1(\norm{a-\tilde{a}}/(2\sqrt{b}))$ where
\begin{eqnarray*}
	\Phi_1(r) &:=& \sqrt{2/\pi} \int_0^r \exp(-x^2/2)\, dx,
\end{eqnarray*}
cf.\ e.g.\ \cite[Exercise 15.12]{MR2807365}. Hence for any $t>0$,
\begin{eqnarray}
	 d(t)&=& ||\law{X_t}-\law{Y_t}||_{TV} \, = \,
	\Phi_1\left(\frac{\norm{m+e^{-t/2}(y-m-x)}}{2\sqrt{1-e^{-t}}}\right ).\label{TV}
\end{eqnarray}
We now compare the upper bound \eqref{TV2} for the total variation distance that
has been derived by sticky couplings to the exact expression \eqref{TV}. Observe
that Assumptions \ref{ass:BoundOnDriftDifference} and \ref{ass:CurvControlForB}
are satisfied with $M=\norm{m}/2$ and the constant function $\kappa(r)=-1/2$
respectively. By a straightforward computation we obtain
\begin{eqnarray}
	\label{pi} \pi[(0,\infty)] &=& 1 \, - \,
	\left(1\,+\,\sqrt{\pi/8}\norm{m} e^{m^2/8}\,(1+\Phi_1(\norm{m}/2))\right)^{-1}.
\end{eqnarray}
Asymptotically as $t\rightarrow\infty$, the upper bound for
$P[\tilde{X}_t\not= \tilde{Y}_t]$ in \eqref{TV2} approaches \eqref{pi}, 
whereas the total variation distance $d(t)$ converges to $\Phi_1(\norm{m}/2)$.
Comparing both expressions for small and large values of $\norm{m}$, we see
that as $\norm{m}\rightarrow 0$,
\begin{eqnarray*}
\pi[(0,\infty)]\sim \sqrt{\pi/8} \norm{m}, \quad\text{whereas} \quad \Phi_1(\norm{m}/2)\sim
\norm{m}/\sqrt{2\pi},
\end{eqnarray*}
and as $\norm{m}\rightarrow \infty$,
$$
	1-\pi[(0,\infty)] \,\sim\, \frac{2}{\sqrt{2\pi}\norm{m}}e^{-\norm{m}^2/8},
	\quad\text{whereas}\quad 1-\Phi_1(\norm{m}/2)\, \sim \,
	\frac{4}{\sqrt{2\pi}\norm{m}}e^{-\norm{m}^2/2}.
$$
Hence as $m\downarrow 0$, the bounds for the long time limit of the total variation distance provided by sticky couplings are of
the correct order up to a multiplicative constant, whereas for $m\to\infty$, we loose a factor $4$
in the exponential.

Furthermore, we can compare the decay rate $\tilde c$ in 
\eqref{TV2} with the rate of convergence of $d(t)$ to
its limit $\Phi_1(\norm{m}/2)$. 
Asymptotically as $t\uparrow\infty$, \eqref{TV} implies
\begin{eqnarray}\nonumber 
	\norm{d(t)-\Phi_1(\norm{m}/2)}&\sim& \Phi_1'(\norm{m}/2) e^{-t/2}
	\norm{y-m-x}/2 \\ &=& (2\pi)^{-1/2} e^{-m^2/8} e^{-t/2}
	\norm{y-m-x}.\label{TVA}
\end{eqnarray}
On the other hand, in this case $\tilde c=1/8$ and $\tilde\epsilon
=1/(2\sqrt 8)$, so by \eqref{TV2}, 
\begin{eqnarray} \label{TV2A}
	P[\tilde{X}_t\not=\tilde{Y}_t] - \pi[(0,\infty)] &\leq& 2^{-1/2}
	(e^{t/8}-1)^{-1} \norm{x-y}.
\end{eqnarray}
We see that the exponential rate of decay in our bound differs from the optimal
rate only by a factor $4$.
\end{example}

\begin{example}[Confined Brownian motion]\label{exa:BM}
Fix $R,k,m\in(0,\infty)$, and let
$$b(x)=0\ \ \mbox{for }\norm{x}\leq R,\quad\mbox{and}\quad b(x)=-k
 (x-R\; {\rm sgn}( {x}))/2\ \ \mbox{otherwise.}$$ 
 Moreover, let $\tilde b(x)=b(x)+m/2$. In this case, Condition \eqref{KLRbound}
 is satisfied with $L=0$, $K=k/6$ and $\mathcal R = 3R$, and
 Assumption \ref{ass:BoundOnDriftDifference} holds with $M=m/2$.
 Assuming $m\le kR$ and $mR\le 4/3$, Theorem \ref{thm2} and the first bound in Lemma \ref{lem:example}  show that there is a
coupling $(\tilde X_t,\tilde Y_t)$ of the corresponding solutions to \eqref{eq:DiffusionDifferentDrifts} and \eqref{eq:DiffusionDifferentDrifts2} with arbitrary initial values $x$ and $y$ such that 
 	\begin{eqnarray}
 		\label{TV8} 
 		\limsup_{t\to\infty} P[\tilde{X}_t\not = \tilde{Y}_t] 
 		&\leq &
 		\left( \frac{3e}{4}R\, +\, \left({3\pi e^3}/{2}\right)^{1/2}\, k^{-1/2}\right)\, m.
 	\end{eqnarray}
On the other hand, the unique invariant probability measures for
\eqref{eq:DiffusionDifferentDrifts} and \eqref{eq:DiffusionDifferentDrifts2} 
are given explicitly by $\nu (dx)=Z_f^{-1}f(x)\, dx$, $\mu (dx)=Z_g^{-1}g(x)\,dx$, respectively, where $f(x)=\exp (-k \max (|x|-R,0)^2/2)$, 
$g(x)=\exp (mx)f(x)$, $Z_f=\int_{-\infty}^\infty f(x)\,dx$ and $Z_g=\int_{-\infty}^\infty g(x)\,dx$. Noting that $Z_g\ge Z_f$, an explicit computation
yields the lower bounds
$$\|\mu -\nu\|_{TV}\ \ge\ (\exp (-mR)-1+mR)/(mR),$$
and, for $Rk^{1/2}\le 1$,
$$\|\mu -\nu\|_{TV}\ \ge\ \left(1-\exp (-mR+m^2/(2k))+2^{1/2}(\pi k)^{-1/2}m\exp (-mR)\right)/4,$$
see the appendix. In particular,
$$\liminf_{m\downarrow 0}\|\mu -\nu\|_{TV}/m\ \ge\ \frac 14\, \left( R+(2/\pi )^{1/2}k^{-1/2}\right) .$$
Hence for small $m$, the bound in \eqref{TV8} is sharp up to a constant
factor. 
\end{example}

\begin{rem}[Comparison with Girsanov couplings]	
An alternative approach to construct couplings of solutions to 
\eqref{eq:DiffusionDifferentDrifts} and \eqref{eq:DiffusionDifferentDrifts2}
is by Girsanov's Theorem. If the initial conditions $X_0$ and $Y_0$ coincide and $T\in [0,\infty )$ is a fixed constant, then Girsanov's Theorem can be applied to construct a coupling 
$(X_s,Y_s)$ such that with positive probability, $X_s=Y_s$ for all $s\in [0,T]$.  Moreover, explicit bounds on this probability can be derived via
Hellinger integrals \cite{MR814659,MR827946,MR1786115}. 
Notice, however, that the corresponding bounds
typically degenerate rapidly as $T\rightarrow\infty$.
Hence Girsanov's Theorem provides a very strong coupling over short
time intervals, whereas the sticky
couplings introduced above are stable for long times in the sense that $\liminf_{t\rightarrow\infty} P[X_t=Y_t] \geq \pi[\{0\}]>0$.
\end{rem}

 \subsection{McKean-Vlasov processes} \label{secMcKeanVlasov}
 
We consider nonlinear diffusions on $\reals^d$ of type
 \begin{eqnarray}\label{nonlinear} 
 	dX_t &=&\eta(X_t) \ dt \ +\ \tau \int \vartheta(X_t,y)\ \mu^x_t(dy) \ dt \ +\  dB_t,\quad X_0=x, \\
 	\mu_t^x &=&\law{X_t},\nonumber
 \end{eqnarray}
 where $(B_t)$ is a $d$-dimensional Brownian motion and $\tau\in\reals$. The SDE is nonlinear in the sense of McKean,
 i.e., the future
development after time $t$ depends on the current state $X_t$ and on the law of $X_t$, cf.\ e.g.\ \cite{MR1108185,MR1431299}.
Let $\eta:\reals^d\rightarrow\reals^d$ and $\vartheta:\reals^d\times\reals^d\rightarrow\reals^d$ 
 be Lipschitz continuous functions. Then
the equation above admits a unique strong solution, cf.\ \cite[Theorem 2.2]{MR1431299}.
Let us fix initial values $x_0,y_0\in\reals^d$, $x_0\not=y_0$, and consider 
solutions $(X_t)$ and $(Y_t)$ of \eqref{nonlinear} with $X_0=x_0$ and $Y_0=y_0$ respectively. We define drift coefficients
\begin{eqnarray}\label{b1}
	b^{x_0}(t,x) &=&  \eta(x) \ +\  \tau  \int \vartheta(x,y) \ \mu_t^{x_0}(dy),
	\\ b^{y_0}(t,x) &=&  \eta(x) \ +\  \tau  \int \vartheta(x,y) \ \mu_t^{y_0}(dy),\label{b2}
\end{eqnarray}
which are uniformly Lipschitz in $x$ and continuous in $t$. Notice that due to pathwise uniqueness,
$(X_t)$ and $(Y_t)$ are the unique strong solutions to the equations
\begin{alignat}{3}\label{a1}
	dX_t &\ = \  b^{x_0}(t,X_t) \ dt &\ +\  dB_t, &\qquad X_0&\ =\ &x_0,\\
	dY_t &\ = \ b^{y_0}(t,Y_t) \ dt &\ +\  dB_t, &\qquad Y_0&\ =\ &y_0,\label{a2}
\end{alignat}
and hence we can interpret the processes as two diffusions with different drifts.
\begin{ass}\label{ass:CurvControlForEta}
 	There is a Lipschitz function $\kappa:[0,\infty)\rightarrow \reals$ such that
 	\begin{equation*}
 		\sProd{x-y}{\eta (x)-\eta (y)} \leq \kappa(\norm{x-y}) \cdot \norm{x-y}^2 \quad\text{for any $x,y\in\reals^d$ and $t\geq 0$}.
 	\end{equation*}
 	Outside of a bounded interval, the function $\kappa$ is constant and strictly negative.
\end{ass}
Assuming that Assumption \ref{ass:CurvControlForEta} holds, we have shown in 
\cite{2016arXiv160606012E} that there are constants
$A,\lambda ,\tau_0\in(0,\infty)$ such that for $\norm{\tau}\le\tau_0$,
\begin{eqnarray}\label{eqmckeanest}
	\wDist^1(\mu^{x}_t,\mu^{y}_t) &\leq & A \ e^{-\lambda\, t}\   |x-y| \quad\text{for any }t\geq 0 \text{ and }x,y\in\reals^d,
\end{eqnarray}
where $\wDist^1$ denotes the standard $L^1$ Wasserstein distance. 
The proof is based on an application of reflection coupling if $|X_t-Y_t|
\geq\delta $ and synchronous coupling if $|X_t-Y_t|
\leq\delta /2$, where $\delta $ is a small positive constant. In the
intermediate region, a combination of both couplings is applied. The bound in \eqref{eqmckeanest} is obtained when considering the limit of the resulting bounds as $\delta\downarrow 0$. The couplings considered 
in \cite{2016arXiv160606012E} now turn out to be approximations of a sticky 
coupling. By applying directly the sticky coupling and 
using Corollary \ref{thmtimedependent} further below, we can extend
the result in \cite{2016arXiv160606012E} and derive a corresponding
exponential decay in total variation norm:
  
\begin{thm}\label{mainMcKeanVlasov}
Let $\eta$ and $\vartheta$ be Lipschitz and let Assumption \ref{ass:CurvControlForEta} be true.
There is $\tau_0\in(0,\infty)$ such that for any $\norm{\tau}\leq \tau_0$ and any $x,y\in\reals^d$ there
are constants $B,c\in(0,\infty)$ such that,
\begin{eqnarray}
	\|{\mu_t^{x}-\mu_t^{y}}\|_{TV} &\leq& B\ e^{-c\, t} \quad\text{for any } t\geq 0.
\end{eqnarray}
\end{thm}
The proof is given in Section \ref{sec:proofs}.

\subsection{Outlook}

The concept of sticky couplings sheds new light onto several results
that have been previously derived using combinations of reflection and
synchronous couplings. A first example of this type has been given in
Theorem \ref{mainMcKeanVlasov}. Without carrying out details,
we mention three further results that probably can be reinterpreted
in terms of sticky couplings:\smallskip

a) {\em Componentwise reflection couplings for interacting diffusions.}
In \cite{Eberle2015}, Wasserstein bounds for interacting diffusions with 
small interaction term (for example of mean-field-type) have been derived by coupling each component 
independently with a reflection coupling if the distance is greater than a given
constant $\delta >0$, and with a synchronous coupling otherwise.  Instead, one could now directly consider a componentwise sticky coupling. As time
evolves, more and more components in this coupling would get stuck at
nearby positions until, after some finite coupling time, all components
coincide. We expect
that such a coupling could be used to derive total variation bounds similar to those in Theorem \ref{mainMcKeanVlasov} for interacting particle systems.\smallskip

b){\em Couplings for infinite-dimensional diffusions.} In
\cite{2016arXiv160507863Z}, Wasserstein contraction rates have been
derived for a class of diffusions on a Hilbert space with possibly degenerate
noise. Here a reflection coupling has been applied to the projection of the process on a finite dimensional subspace, whereas the remaining (orthogonal) components have been coupled 
synchronously. Again, because of the interaction between the components,
the reflection coupling is switched off when the finite dimensional 
projections of the two copies are close to each other.  Similarly as above,
it should be possible to replace the coupling for the finite dimensional 
projection by a sticky coupling. The resulting infinite dimensional coupling
process would then spend a certain amount of time at states where the
finite dimensional projections of the two copies coincide. 
Under the assumptions made in \cite{2016arXiv160507863Z},  the orthogonal infinite dimensional components would approach each other for large $t$,
and, consequently, the finite dimensional projections would coincide
for an increasing proportion of time.\smallskip   

c) {\em Couplings for Langevin processes.}
In a forthcoming paper, we consider couplings for (kinetic) Langevin
diffusions $(X_t,V_t)_{t\ge 0}$
with state space $\mathbb R^{2d}$ that are given by
stochastic differential equations of type
\begin{eqnarray}
\label{Langevin}dX_t &=&V_t\, dt,\\
\nonumber dV_t&=&-\gamma V_t\, dt\, -\,  u \, \nabla U(X_t)\, dt\,
+\, \sqrt{2\gamma u}\, dB_t.
\end{eqnarray}
Here 
$(B_t)_{t\ge 0}$ is a $d$ dimensional Brownian motion, 
$u$ and $\gamma $ are positive constants, and $U$ is a $C^1$ 
function on $\mathbb R^d$. We apply a reflection coupling that is 
replaced by a synchronous coupling when the values of
$X_t+\gamma^{-1}V_t$ are close to each other for both components.
Again, at least informally, this coupling could be replaced by
a coupling $((X_t,V_t),(X_t^\prime ,V_t^\prime ))$ that is sticky 
when $X_t+\gamma^{-1}V_t=X_t^\prime+\gamma^{-1}V_t^\prime$.
Under the assumptions that we impose on $U$, the coupling would be contractive on the corresponding $3d$ dimensional linear subspace of $\mathbb R^{4d}$, and as time evolves, it would spend
a positive amount of time on this subspace.\medskip

We hope that the potential applications listed above show how
sticky couplings provide a valuable concept for building intuition
about ways to couple diffusion processes in an efficient way. 
Carrying out carefully the ideas described above would go
far beyond the scope of this paper.

\section{Diffusions on $\reals_+$ with a sticky reflecting boundary.}\label{sec:Sticky}

In this section we prove some basic results on diffusions on $\mathbb R_+$ with a sticky boundary at $0$. In particular,
we prove the existence of a synchronous coupling
of two sticky diffusions and a corresponding comparison theorem,
which is then applied to study the long time behavior of the 
processes. At first, we need to adapt some known facts on
existence and uniqueness of weak solutions to our setup.
We consider the stochastic differential equation 
\begin{eqnarray}\label{eq:oneDimSticky}
	d r_t & = & \alpha(t,r_t)\ dt \ + \ 2 \ I{(r_t>0)} \ dW_t, \qquad \law{r_0} \ = \ \init,
\end{eqnarray}
on the positive real line $\reals_+=[0,\infty)$, where $(W_t)$ is a one-dimensional Brownian motion and
$\init$ is a probability measure on $\reals_+$. Below, we will impose conditions on the drift coefficient
$\alpha:\reals_+\times\reals_+\rightarrow \reals$ which imply existence and uniqueness of weak solutions.
In particular, we will assume that $\alpha(t,0)>0$ for any $t\geq 0$. Let us briefly discuss
the consequences of this assumption: Suppose that $(r_t)$ is a solution of
\eqref{eq:oneDimSticky}. An application of the It{\^o}-Tanaka
formula to $f(r_t)$ with the function $f(x)=\max(0,x)$ and a comparison with
\eqref{eq:oneDimSticky} shows that almost surely,
\begin{eqnarray}\label{eq:stickiness}
	 \int_0^t \alpha(s,0)\ I(r_s=0)\  ds  & = &  \frac{1}{2}\ell_t^0(r), \qquad 0\ \leq \ t \ < \ \infty,
\end{eqnarray}
where
$\ell_t^0(r)=\lim_{\epsilon\downarrow 0} \epsilon^{-1} \int_0^t I(0\leq r_s \leq \epsilon)\, d\quadV{r}_s$
is the right local time of $(r_t)$. 
Equation \eqref{eq:stickiness} shows that there is \emph{reflection} at zero. Moreover, for almost all trajectories,
the Lebesgue measure of the set $\{0\leq s\leq t: r_s=0\}$ increases whenever $\ell_t^0(r)$ increases.
In this sense $(r_t)$ is \emph{sticky} at zero.  
\par\smallskip 
Stochastic differential equations with boundary conditions have a long history. The discovery of a sticky boundary behavior for
one-dimensional diffusions seems to go back to Feller \cite{MR0063607,MR0065809}.
A historical overview is given in \cite{peskir2014boundary}. We give references to the most
relevant works for our application and some recent developments. Existence and uniqueness results for multi-dimensional diffusion processes
with various boundary behaviors have been established by Ikeda and Watanabe in \cite{MR0126883,MR0275537,MR0287612}. These are based on
results by Skorokhod and  McKean 
\cite{skorokhod1961stochastic,skorokhod1962stochastic,mckean1963}. Martingale problems with boundary conditions have been investigated
by Stroock and Varadhan \cite{MR0277037}, see also the related work \cite{MR937956}. Non-existence of a strong solution to the SDE for sticky Brownian motion has been established in \cite{MR1639096}. In \cite{MR1478711},
Warren identifies the law of a sticky Brownian motion conditioned on the driving Wiener process, see also the related work \cite{2016arXiv160307456H}. 
A  recent publication on existence and uniqueness, which is also a good
introduction into the topic, is the work by Engelbert and Peskir \cite{MR3271518} and the related work \cite{MR3183576}. First steps towards sticky couplings
in a one-dimensional setting have been made by Howitt in \cite{Howitt} based on time-changes. The recent articles \cite{2014arXiv1412.3975G,2014arXiv1410.6040G} 
use Dirichlet forms to investigate sticky diffusions and provide some ergodicity results. Rácz and Shkolnikov \cite{MR3325271} construct a multi-dimensional
sticky Brownian motion as a limit of exclusion processes, see also \cite{MR1136247} and \cite{MR598937}.

\subsection{Existence, uniqueness and comparison of solutions}
We use the concept of weak solutions.
Let $(\Omega,\mathcal{A},(\filt_t),P)$ be a filtered probability space satisfying
the usual conditions. An $(\filt_t)$ adapted process 
$(r_t,W_t)$ on $(\Omega,\mathcal{A},P)$ is called a \emph{weak solution} of \eqref{eq:oneDimSticky} if
$P\circ{r_0^{-1}} = \init$, $(W_t)$ is a one-dimensional $(\filt_t)$-Brownian motion w.r.t.\ $P$, and
$(r_t)$ is continuous, non-negative, and $P$-almost surely,
  \begin{equation*}
  r_t-r_0 \ = \  \int_0^t \alpha(s,r_s)\ ds \ + \ \int_0^t 2\ I(r_s>0) \  dW_s, \qquad 0 \ \leq \  t \ < \  \infty.
  \end{equation*}
 We will make the following assumptions on the drift coefficient:
\begin{ass}\label{ass:positiveAtZero}
For any $R>0$, $\inf_{t\in[0,R]} \alpha(t,0) \ >\  0$.
\end{ass}
\begin{ass}\label{ass:stickySdeDriftLocallyLipschitzAndPositive}
For any $R>0$ there is $L_R\in(0,\infty)$ such that
	\begin{eqnarray*}
		 \norm{\alpha(t,x)-\alpha(s,y)} & \leq & L_R \ \left(\ \norm{t-s} \ +\ \norm{x-y}  \ \right) \quad \text{for any } x,y,s,t\in[0,R].
	\end{eqnarray*}
\end{ass}
\begin{ass}\label{ass:stickySdeNonExplosiveCondition}
There is $C\in(0,\infty)$ such that {for any }$x\in\reals_+$,
\begin{eqnarray*}
	\sup\nolimits_{t\in[0,\infty)} \alpha(t,x) &\leq& C \ (\ 1\ +\ \norm{x}\ ) \quad
\end{eqnarray*}
\end{ass}
The assumptions above imply existence and uniqueness in law of weak solutions
to \eqref{eq:oneDimSticky}. This has been proven by Watanabe in \cite{MR0275537,MR0287612}
assuming that the maps $(t,x)\mapsto \alpha(t,x)$ and $t\mapsto 1/\alpha(t,0)$
are bounded and Lipschitz. Using localization techniques
for martingale problems, following the work
of Stroock and Varadhan \cite{MR2190038}, Watanabe's results can be transferred to our slightly more general setup:

\subsubsection{Uniqueness in law}\label{sec:uniqueness}
Let $\mathbb{W}=C(\reals_+,\reals)$ be the space of continuous functions endowed with the topology of
uniform convergence on compacts, and let $\mathcal{B}(\mathbb{W})$ denote the Borel $\sigma$-Algebra. Let $\filt_t=\sigma(\boldsymbol{r}_s: 0\leq s\leq t)$
be the natural filtration generated by the canonical process
$\boldsymbol{r}_t(\omega )=\omega (t)$.
Given a solution $(r_t)$ of \eqref{eq:oneDimSticky},
defined on a probability space $(\Omega,\mathcal{A},P)$,
we write $\mathbb{P}=P\circ r^{-1}$ for the law of $r$ on $(\mathbb{W},\mathcal{B}({\mathbb{W}}))$.
We say that solutions to \eqref{eq:oneDimSticky} are \emph{unique in law}, if
any two solutions $(r_t^1)$ and $(r_t^2)$ with coinciding initial law
 have the same law on the space $(\mathbb{W},\mathcal{B}({\mathbb{W}}))$.
\par\smallskip
In order to apply existing localization techniques for martingale problems, we
interpret equation \eqref{eq:oneDimSticky} as an equation on $\reals$, instead of $\reals_+$, setting
$\alpha(t,x)=\alpha(t,0)$ for $x<0$. This does not cause any problems since, under the assumptions imposed above, any solution $(r_t)$
with initial law supported on $\reals_+$ satisfies almost surely $r_t\geq 0$ for all $t\geq 0$, see e.g.\ the argument in \cite[Proof of Theorem 5]{MR3271518}.
\par\smallskip
We follow \cite{MR2190038,MR1121940}  and define a family of second order differential operators
\begin{eqnarray*}
	(\mathcal{L}_t f)(x) &=& \alpha(t,x) \  f'(x)\ + \ ({1}/{2}) \ I(x>0) \  f''(x).
\end{eqnarray*}
A probability measure $\mathbb{P}$ on $(\mathbb{W},\mathcal{B}({\mathbb{W}}))$ is called a \emph{solution to the martingale problem
w.r.t.\ $(\generator_t)$} iff for any $f\in C_0^2(\reals)$,
\begin{eqnarray*}
 M_t^f & = & f(\boldsymbol{r}_t) -f(\boldsymbol{r}_0) - \int_0^t (\mathcal{L}_{u} f)(\boldsymbol{r}_u) \ du
\end{eqnarray*}
is a continuous $(\filt_t)$-martingale under $\mathbb{P}$. The solution to the martingale problem is
called \emph{unique}, if any two solutions $\mathbb{P}^1$
and $\mathbb{P}^2$ coincide whenever $\mathbb{P}^1\circ \boldsymbol{r}_0^{-1} =\mathbb{P}^2\circ \boldsymbol{r}_0^{-1}$. The next two results are
well-known:

\begin{lem}\cite{MR2190038,MR1121940}\label{lem:MartingaleProblem}
	The following statements are
	 equivalent:
	\begin{itemize}
	  \item[(i)] There is a weak solution of \eqref{eq:oneDimSticky} with initial distribution $\mu$.
	  \item[(ii)] There is a solution $\mathbb{P}$ to the martingale problem w.r.t.\ $(\generator_t)$ s.t.\
	  $\mathbb{P}\circ \boldsymbol{r}_0^{-1}=\mu$.
	\end{itemize}
	Moreover, the uniqueness of solutions to the martingale problem w.r.t.\ $(\mathcal{L}_{t})$ and
	the uniqueness in law of weak solutions to \eqref{eq:oneDimSticky} are equivalent.
\end{lem}

\begin{lem}\cite{MR0275537,MR0287612}\label{existenceWatanabe}
	Assume that the maps $(t,x)\mapsto \alpha(t,x)$ and $t\mapsto 1/\alpha(t,0)$ are bounded and Lipschitz. Then for any
	initial law $\mu$ on $\reals_+$, there is a weak solution
	to \eqref{eq:oneDimSticky} which is unique in law.
\end{lem}

	A detailed proof of Lemma \ref{lem:MartingaleProblem} can be found in \cite[Chapter 5, Section 4.B]{MR1121940}. A proof of Lemma 
	\ref{existenceWatanabe} is given in \cite[Chapter IV, Section
	7]{MR1011252}.

\begin{lem}\label{lem:uniqueness}
	If Assumptions \ref{ass:positiveAtZero} and \ref{ass:stickySdeDriftLocallyLipschitzAndPositive} are satisfied then
	the solution to the martingale problem w.r.t.\ $(\mathcal{L}_{t})$ is unique
	for a given initial law,
 	and thus uniqueness in law holds for solutions to \eqref{eq:oneDimSticky}.
\end{lem}

\begin{proof}
We set $\alpha_n(s,x)=\alpha(s\wedge n,x\wedge n)$ for $n\in\naturals$.
By the assumptions, the maps $(t,x)\mapsto \alpha_n(t,x)$ and $t\mapsto 1/\alpha_n(t,0)$ are
bounded and Lipschitz continuous.
Hence uniqueness holds for the corresponding martingale problem
for any initial law $\mu$ on $\reals_+$ according to Lemma
\ref{existenceWatanabe} and \ref{lem:MartingaleProblem}.
The uniqueness for the martingale problem w.r.t.\ $(\mathcal{L}_{t})$
for such initial laws can now be shown by a localization argument,
cf.\ \cite[Theorem 10.1.2]{MR2190038}.
\end{proof}

\subsubsection{Approximation, existence and coupling of solutions}\label{sec:stickyApproxAndExist}
We now consider two equations of the form \eqref{eq:oneDimSticky} with 
drift coefficients $\beta$ and $\gamma$ that both satisfy Assumptions
 \ref{ass:positiveAtZero},  \ref{ass:stickySdeDriftLocallyLipschitzAndPositive} and \ref{ass:stickySdeNonExplosiveCondition}. We construct a synchronous 
 coupling of solutions to these equations as a weak limit of solutions to
 approximating equations with locally Lipschitz continuous coefficients.
We introduce the family of
stochastic differential equations, indexed by $n\in\naturals$, given by
\begin{eqnarray}\label{eq:stickApprox1}
	d\tilde{r}^n_t & = & \beta(t,\tilde{r}^n_t) \ dt \ +\ 2\ \vartheta^n(\tilde{r}^n_t) \ d\tilde{W}_t, \qquad \operatorname{Law}(\tilde{r}^n_0,\tilde{s}^n_0)=\tilde{\mu}^n\otimes\tilde{\nu}^n,\\
\nonumber d\tilde{s}^n_t & = & \gamma(t,\tilde{s}^n_t) \ dt \ + \ 2\  \vartheta^n(\tilde{s}^n_t) \ d\tilde{W}_t,
\end{eqnarray}
Here $(\tilde{W}_t)$ is a Brownian motion, and we assume that:
\begin{ass}\label{ass:initialConvergence}
$(\tilde{\mu}^n)$ and $(\tilde{\nu}^n)$ are sequences of probability measures on $\reals_+$ converging weakly towards probability measures $\tilde{\mu}$
and $\tilde{\nu}$, respectively.
\end{ass}
\begin{ass}\label{ass:DiffusionCoefficientApproximation}
	For each $n\in\mathbb{N}$, the function $\vartheta^n:\reals_+\rightarrow[0,1]$ is Lipschitz continuous with
	$\vartheta^n(0)=0$, $\vartheta^n(x)>0$ for $x>0$, and $\vartheta^n(x)=1$ for $x\geq 1/n$.
\end{ass}

\begin{rem}
In \cite{MR3271518}, a sticky Brownian motion $(r_t)$ satisfying
\begin{eqnarray*}
	dr_t &=&  I(r_t\not= 0) \, d\tilde{W}_t, \qquad I(r_t=0)\, \mu \, dt \, = \,
	d\ell_t^0(r), \qquad \mu\in(0,\infty),
\end{eqnarray*}
is approximated  by solutions of equations
	\begin{eqnarray*}
		dr^n_t &=& \left( \sqrt{2\,\mu/n}\ I(\norm{r^n_t}\leq 1/n) \ + \ I(\norm{r^n_t}>1/n)\right)d\tilde{W}_t,
	\end{eqnarray*}
	The approximation is tailored in such a way that it is compliant with
	the time-changes frequently used to show existence and uniqueness of weak solutions to sticky SDEs, see e.g.\ \cite{MR3271518,MR0287612}. Our approximation
	result follows a similar spirit but it does not rely on time changes.
\end{rem}

\begin{lem}\label{lemExistenceApproximationEquation}
	Suppose that $\beta$ and $\gamma$ satisfy Assumptions \ref{ass:positiveAtZero}, \ref{ass:stickySdeDriftLocallyLipschitzAndPositive} and \ref{ass:stickySdeNonExplosiveCondition}.
	Moreover, let Assumptions \ref{ass:initialConvergence} and \ref{ass:DiffusionCoefficientApproximation} be true.
	Then for each $n\in\mathbb{N}$, there is a strong solution
	$(\tilde{r}^n_t,\tilde{s}^n_t)$ of
	Equation \eqref{eq:stickApprox1} with values in $\reals_+^2$. Moreover, uniqueness in law holds.
\end{lem}

\begin{proof}
	Fix $n\in\naturals$. For $x<0$ we set $\vartheta^n(x)=0$, $\beta(t,x)=\beta(t,0)$, and $\gamma(t,x)=\gamma(t,0)$.
	Equation \eqref{eq:stickApprox1} is then a standard SDE on $\reals^2$ with locally Lipschitz coefficients.
	Hence there is a strong
	 and pathwise unique solution. 
	 Moreover, Assumption \ref{ass:stickySdeNonExplosiveCondition} implies that the solution
	is non-explosive. Similarly to  \cite[Proof of Theorem 5]{MR3271518}, we can apply the It{\^o}-Tanaka formula to
	 the negative part of $\tilde{r}^n_t$ in order to show that
	the process is non-negative. Indeed,
	\begin{eqnarray*}
		(\tilde{r}_t^{n})^--(\tilde{r}^{n}_0)^- & = & - \int_0^t I(\tilde{r}^n_s\leq 0) \, d\tilde{r}_s^n + \frac{1}{2} \ell_t^{0}(\tilde{r}^n),
	\end{eqnarray*}
	where $\ell_t^0(\tilde{r}^n)$ is the right local time of $(\tilde{r}^n_t)$, i.e., 
$$
		\ell_t^0(\tilde{r}^n) = \lim_{\epsilon\downarrow 0} \epsilon^{-1} \int_0^t I(0\leq \tilde{r}^n_s \leq \epsilon)\, d\quadV{\tilde{r}^n}_s
		 = 4\lim_{\epsilon\downarrow 0} \epsilon^{-1} \int_0^t I(0\leq \tilde{r}^n_s \leq \epsilon)\, \vartheta^n(\tilde r^n_s)^2\, ds .
	$$
Since $\vartheta^n$ is Lipschitz with $\vartheta^n(0)=0$, the local time vanishes. Therefore, and since $\beta(s,0)>0$ for any $s\geq 0$, we have $0\leq (\tilde{r}_t^{n})^-\leq (\tilde{r}_0^{n})^- = 0$. A similar argument can be used
	for $(\tilde{s}_t^n)$.
\end{proof}

For each $n\in\naturals$, there are a probability space $(\Omega^n,\mathcal{A}^n,P^n)$ and random variables
$\tilde{r}^n,\tilde{s}^n:\Omega^n\rightarrow \mathbb{W}$ such that
$(\tilde{r}^n_t,\tilde{s}^n_t)$ is a solution of  \eqref{eq:stickApprox1}. Let $\mathbb{P}^n=P^n\circ(\tilde{r}^n,\tilde{s}^n)^{-1}$ denote the law on
$\mathbb{W}\times \mathbb{W}$.
For $w=(w_1,w_2)\in \mathbb{W}\times \mathbb{W}$,
we define the coordinate mappings $\boldsymbol{r}(w)=w_1$ and $\boldsymbol{s}(w)=w_2$.

\begin{thm}\label{thmApproximation}
Suppose that $\beta$ and $\gamma$ satisfy Assumptions \ref{ass:positiveAtZero},
\ref{ass:stickySdeDriftLocallyLipschitzAndPositive} and \ref{ass:stickySdeNonExplosiveCondition}, and let $\tilde\mu$ and $\tilde\nu$ be probability measures on $\mathbb R_+$. Suppose that the sequences $(\vartheta^n)$,
$(\tilde{\mu}^n)$ and $(\tilde{\nu}^n)$ satisfy Assumptions \ref{ass:initialConvergence} and \ref{ass:DiffusionCoefficientApproximation}.
Then there is a random variable
$(\tilde{r},\tilde{s})$ with values in $\mathbb{W}\times\mathbb{W}$, defined on some probability space  $(\Omega, \mathcal{A}, P)$,
such that $(\tilde{r}_t,\tilde{s}_t)$ is a weak solution of
\begin{eqnarray}\label{eq:thmApproximation}
	d \tilde{r}_t & = & \beta(t,\tilde{r}_t)\ dt \ + 2 \ I{(\tilde{r}_t>0)} \ d\tilde{W}_t, \qquad \law{\tilde{r}_0,\tilde{s}_0} \ = \ \tilde{\mu}\otimes\tilde{\nu},\\
	\nonumber d \tilde{s}_t & = & \gamma(t,\tilde{s}_t)\  dt \   + 2 \ I{(\tilde{s}_t>0)} \ d\tilde{W}_t,
\end{eqnarray}
for some Brownian motion $(\tilde{W}_t)$. Moreover,
there is a subsequence $(n_k)$ such that
$P^{n_k}\circ(\tilde{r}^{n_k},\tilde{s}^{n_k})^{-1}$ converges
weakly towards
$P\circ(\tilde{r},\tilde{s})^{-1}$. If additionally,
\begin{eqnarray}\label{m1}
	\beta(t,x) \ \leq\  \gamma(t,x) \quad &&\text{for any } x,t\in\reals_+ ,\text{ and }\\\label{m2}
	P^{n}[\ \tilde{r}^n_0\leq \tilde{s}^n_0 \ ]\ =\ 1  \quad &&\text{for any } n\in\naturals,
\end{eqnarray}
then
$P[\ \tilde{r}_t \leq \tilde{s}_t \text{ for all }t\geq 0\ ]=1$.
\end{thm}

\begin{proof}
We 
fix sequences of diffusion coefficients $(\vartheta^n)$ and initial conditions $(\tilde{\mu}^n)$ and $(\tilde{\nu}^n)$
satisfying Assumptions \ref{ass:initialConvergence} and \ref{ass:DiffusionCoefficientApproximation}.
\par\smallskip 
\emph{Tightness:}
We claim that the sequence $(\mathbb{P}^{n})_{n\in\naturals}$ of probability measures on $(\mathbb{W}\times\mathbb{W},
\mathcal{B}(\mathbb{W})\otimes\mathcal{B}(\mathbb{W}))$ is
tight. This can be shown by similar arguments as in \cite{MR2870529,MR3175790}, so we only explain briefly how to
adapt these arguments to our setting. At first, we observe that
a uniform Lyapunov condition holds for the Markov processes
$(\tilde{r}^n_t,\tilde{s}^n_t)$ defined by \eqref{eq:stickApprox1}.
Indeed, these processes solve a local martingale problem w.r.t.\ the
generators
\begin{equation}
\label{eq:Martpro}
\mathcal L_t^n\ =\ \beta(t,\cdot )\,\partial_r\, +\, \gamma (t,\cdot )\,
\partial_s\, +\, 2(\vartheta^n)^2\, (\partial^2_r+\partial^2_s)
\end{equation}
defined on smooth functions on $\mathbb R^2$. Let
$V(x):=1+\norm{x}^2$ for $x\in\reals^2$. Recall that the drift
coefficients in \eqref{eq:Martpro} do not depend on $n$ and that they
satisfy the linear growth Assumption \ref{ass:stickySdeNonExplosiveCondition}. Moreover,
the diffusion coefficients are uniformly bounded by one. It follows that there
is a constant $\lambda\in(0,\infty)$, not depending on $n$, such that $\generator^n_t V\leq \lambda V$ for any $n\in\mathbb N$. From this
one can conclude that for each finite time interval $[0,T]$ and every
$\epsilon >0$, there is a compact set $K\subseteq\mathbb R^2$ such that for any $n\in\mathbb N$,
$P[(\tilde{r}^n_t,\tilde{s}^n_t)\in K\mbox{ for }t\le T]\ge 1-\epsilon$. Moreover, the drift and
diffusion coefficients are uniformly bounded on the set $K$.  Combining these
arguments, we can conclude tightness of the laws on $\mathbb W\times
\mathbb W$. We refer to \cite{MR2870529,MR3175790} for a detailled proof in a similar setting. By Prokhorov's Theorem, 
we can conclude that there is a subsequence $n_k\rightarrow\infty$ 
and a probability measure $\mathbb{P}$ on $\mathbb W\times\mathbb W$ such that $\mathbb{P}^{n_k}\rightarrow \mathbb{P}$ weakly. To simplify notation we will write in the following
$n$ instead of $n_k$, keeping in mind that we have convergence only along a
subsequence.
\par\smallskip
\emph{Identification of the limit:}
We now characterize the measure $\mathbb{P}$. In principle, we follow well-known strategies for identifying limits
of semimartingales, cf.\ \cite{MR2190038,MR1943877,MR838085}.
However, we can not apply those results directly, because the diffusion coefficients in \eqref{eq:thmApproximation}
are discontinuous.
\par\smallskip
We know that
$\mathbb{P}\circ (\bs{r}_0,\bs{s}_0)^{-1}=\mu\otimes\nu$, since $\mathbb{P}^n\circ (\bs{r}_0,\bs{s}_0)^{-1}=\mu^n\otimes\nu^n$
converges weakly to $\mu\otimes\nu$ by assumption.
We define maps $\bs{M},\bs{N}:\mathbb{W}\times\mathbb{W}\rightarrow \mathbb{W}$ by
\begin{equation*}
\bs{M}_{t}=\bs{r}_t-\bs{r}_0-\int_0^{t} \beta(u,\bs{r}_u)\, du \quad\text{and}\quad \bs{N}_{t}=\bs{s}_t-\bs{s}_0-\int_0^{t} \gamma(u,\bs{s}_u)\, du.
\end{equation*}
We claim that
$(\bs{M}_{t}, \filt_t, \mathbb{P})$ and $(\bs{N}_{t}, \filt_t, \mathbb{P})$ are martingales w.r.t.\ the canonical filtration $\filt_t=\sigma((\bs{r}_u,\bs{s}_u)_{0\leq u\leq t})$.
Indeed, the mappings $\bs{M}$ and $\bs{N}$ are continuous
on $\mathbb{W}$, so by the continuous mapping theorem,
$\mathbb{P}^n\circ (\bs{r},\bs{s},\bs{M},\bs{N})^{-1}$ converges weakly to $\mathbb{P}\circ (\bs{r},\bs{s},\bs{M},\bs{N})^{-1}$.
Notice that for each $n\in\naturals$,  $(\bs{M}_t, \filt_t, \mathbb{P}^n)$
is a martingale. Moreover, for any fixed $t\geq 0$,
the family $(\bs{M}_t, \mathbb{P}^n)_{n\in\naturals}$ is uniformly integrable. 
Hence
$(\bs{M}_{t}, \filt_t, \mathbb{P})$ is a continuous martingale, cf.\ \cite[Chapter IX, Proposition 1.12]{MR1943877}. In particular,
the quadratic variation $(\quadV{\bs{M}}_t)$ exists $\mathbb P$-almost surely. 
Notice that, by \eqref{eq:stickApprox1}, $[\bs M]_t\le 4t$ $\mathbb P^n$-almost surely for every $n$. Thus
for any $t\geq 0$, 
\begin{align*}
	\mathbb{E}\left[\sup_{0\leq s \leq t} \norm{\bs{M}_s}^2\right]
	&\leq \liminf_{R\rightarrow\infty} \mathbb{E}\left[\sup_{0\leq s \leq t} \norm{\bs{M}_s}^2\wedge R\right]
	= \liminf_{R\rightarrow\infty} \lim_{n\rightarrow\infty} \mathbb{E}^{n}\left[\sup_{0\leq s \leq t} \norm{\bs{M}_s}^2\wedge R\right] \nonumber
	\\&\leq \liminf_{n\rightarrow\infty} \mathbb{E}^{n}\left[\sup_{0\leq s \leq t} \norm{\bs{M}_s}^2\right] \nonumber
	\ \leq\   4\ \liminf_{n\rightarrow\infty} \mathbb{E}^{n}\left[\;\quadV{\bs{M}}_t\;\right] \ \leq\ 16 \, t,\nonumber
\end{align*}
Hence, under 
$\mathbb{P}$, $(\bs{M}_{t})$ is a
square integrable martingale, and thus $(\bs{M}_{t}^2-\quadV{\bs{M}}_t)$
is a martingale, cf.\ \cite[Theorem 21.70]{MR3112259}. Similar statements hold for $(\bs{N}_{t})$.
\par\smallskip
As a next step, we compute the quadratic variations and covariations of
 $(\bs{M}_{t})$ and  $(\bs{N}_{t})$ under $\mathbb P$. 
Here we follow arguments from \cite{MR3325271}.
Similarly as above, the family  $(\bs{M}_t^2,\mathbb{P}^n)$ is uniformly integrable for any fixed $t\geq 0$,
i.e.,
\begin{eqnarray}\label{eq:unifo}
	\lim_{\delta\rightarrow\infty} \, \sup_{n\in\naturals} \,\mathbb{E}^n[\,\norm{\bs{M}_t}^2;\, \norm{\bs{M}_t}^2>\delta\,] &=& 0.
\end{eqnarray}
Indeed, by Burkholder's inequality, there is a constant $C\in(0,\infty)$
such that
\begin{eqnarray*}
	\mathbb{E}^n\left[\,\bs{M}_t^4\,\right]&\leq& C \ \mathbb{E}^n\left[\,\quadV{\bs{M}}^2_t\,\right] \; \leq \ 16\ C\ t^2\qquad\mbox{for any }n\in\mathbb N.
\end{eqnarray*}
Let $G:\mathbb{W}\rightarrow \reals_+$ be bounded, continuous and non-negative. Equation \eqref{eq:unifo} implies
$$
	 \lim_{\delta\rightarrow\infty}  \sup_{n\in\naturals} \mathbb{E}^n\left[ |G \bs{M}_t^2-G (\bs{M}_t^2\wedge \delta)|\right]
\ \leq\ \norm{G}_\infty \liminf_{\delta\rightarrow\infty}  \sup_{n\in\naturals} \mathbb{E}^n\left[\bs{M}_t^2; \bs{M}_t^2>\delta\right]\ =\ 0.
$$
Hence for any such $G$ and any $t\geq 0$,
\begin{eqnarray}\label{eqappro}
	\mathbb{E}\left[G\, \bs{M}_t^2\right] &=& \lim_{\delta\rightarrow\infty} \mathbb{E}\left[G\, (\bs{M}_t^2 \wedge \delta)\right]
	\;=\; \lim_{\delta\rightarrow\infty} \lim_{n\rightarrow\infty} \mathbb{E}^n\left[G\, (\bs{M}_t^2 \wedge \delta)\right]
	\\\nonumber&=&\lim_{n\rightarrow\infty} \lim_{\delta\rightarrow\infty}  \mathbb{E}^n\left[G\, (\bs{M}_t^2 \wedge \delta)\right]
	=\lim_{n\rightarrow\infty} \mathbb{E}^n\left[G\, \bs{M}_t^2\right].
\end{eqnarray}
We now show that
$(\bs{M}^2_t-4\int_0^t I(\bs{r}_u>0)\,du, \mathbb{P})$ is a submartingale.
Fix $0\leq s < t$. Then for any continuous, bounded  and $\filt_{s}$-measurable function $G:\mathbb{W}\rightarrow \reals_+$,
\begin{align}\label{eq2}
	\lim_{n\rightarrow\infty} \mathbb{E}^n\left[G\int_s^t 4\, \vartheta^n(\bs{r}_u)^2\, du\right]
	=\lim_{n\rightarrow\infty} \mathbb{E}^n\left[G\left(\bs{M}^2_t-\bs{M}^2_s\right)\right]
	=\mathbb{E}\left[G\left(\bs{M}^2_t-\bs{M}^2_s\right)\right].
\end{align}
 On the other hand, the map $w\mapsto \int_0^{\cdot }\, I(w_s>\epsilon) \,ds$ from $\mathbb{W}$ to $\mathbb{W}$ is lower semicontinuous
 for any $\epsilon\geq 0$. Fatou's lemma and the Portemanteau theorem imply
 \begin{eqnarray}\label{eq1}
 	\mathbb{E}\left[G\, \int_s^t I(\bs{r}_u>0 )\, du \right] &\leq& \liminf_{\epsilon\downarrow 0} \mathbb{E}\left[G\, \int_s^t I(\bs{r}_u
 	>\epsilon )\, du \right]
 	\\\nonumber &\leq& \liminf_{\epsilon\downarrow 0}\, \liminf_{n\rightarrow\infty}  \mathbb{E}^n\left[G\, \int_s^t I(\bs{r}_u>\epsilon )\, du \right].
 \end{eqnarray}
 Notice that for any fixed $\epsilon>0$,
\begin{equation} \label{eq3}
	\liminf_{n\rightarrow\infty} \mathbb{E}^n\left[G\,\left(\int_s^t \vartheta^n(\bs{r}_u)^2\, du- \int_s^t I(\bs{r}_u>\epsilon )\, du\right) \right]\geq 0.
\end{equation}
By \eqref{eq2}, \eqref{eq1} and \eqref{eq3}, we have
$$
	\mathbb{E}\left[G\left(\bs{M}^2_t-\bs{M}^2_s-4 \int_s^t I(\bs{r}_u>0)\, du\right)\right]\ \geq\  0.
$$
Invoking a monotone class argument, cf.\ \cite[Theorem 8]{MR2273672}, we see that $(\bs{M}^2_t-4 \int_0^t I(\bs{r}_s>0)\,ds, \filt_t, \mathbb{P})$ is indeed a submartingale.
We show that it is also a supermartingale and hence a martingale.
By \eqref{eqappro}, for any function $G$ as above,
\begin{eqnarray*}
	\mathbb{E}\left[G\left(\bs{M}^2_t-\bs{M}^2_s-4\,(t-s)\right)\right]
&=&\lim_{n\rightarrow\infty} \mathbb{E}^n\left[G\left(\bs{M}^2_t-\bs{M}^2_s-4\,(t-s)\right)\right] \ \leq \  0
\end{eqnarray*}
Hence, $\bs{M}_t^2-4\,t$ is a supermartingale under $\mathbb{P}$.
 The uniqueness of the Doob-Meyer decomposition \cite[Theorem 16]{MR2273672}
implies that the map $t\mapsto [\bs{M}]_t-4\,t$ is $\mathbb{P}$-almost surely decreasing.
Observe that
$(\bs{r}_t,\filt_t,\mathbb{P})$ is a continuous semimartingale with $\quadV{\bs{r}}=\quadV{\bs{M}}$. Hence
the It{\^o}-Tanaka formula
implies that $\mathbb{P}$-almost surely,
\begin{eqnarray}\label{d1}
		\int_0^t I(\bs{r}_u=0) d\quadV{\bs{M}}_u &=& \int_0^t I(\bs{r}_u=0)  d\quadV{\bs{r}}_u =\int_{-\infty}^{\infty} I(y=0) \ell_t^y(\bs{r})  dy = 0.
\end{eqnarray}
We  conclude that for any $0\leq s<t$,
$$
[\bs{M}]_t-[\bs{M}]_s\ =\	\int_s^t I(\bs{r}_u>0)\,d[\bs{M}]_u  \ \leq\ 4\int_s^t I(\bs{r}_u>0)\,du,
$$
and hence for any $\mathcal F_s$-measurable function $G\in C_b(\mathbb W)$,
\begin{eqnarray*}
&&\mathbb{E}\left[G\left(\bs{M}^2_t-\bs{M}^2_s-4 \int_s^t I(\bs{r}_u>0)\,du\right)\right]\leq 0.
\end{eqnarray*}
As above we conclude by a monotone class argument that $(\bs{M}_t^2-4 \int_0^t I(\bs{r}_u>0)\,du)$ is a supermartingale, and hence a martingale, i.e.,
\begin{equation}\label{eq:qvM}
[ \bs M]\ =\ 4\, \int_0^\cdot I(\bs r_u>0)\, du\qquad\mathbb P\mbox{-almost surely.}
\end{equation}
 Similarly, we can show that
 \begin{equation}\label{eq:qvN}
[ \bs N]\ =\ 4\, \int_0^\cdot I(\bs s_u>0)\, du\qquad\mathbb P\mbox{-almost surely.}
\end{equation}
Moreover, we claim that 
\begin{eqnarray}\label{c1}
	\coV{\bs{M}}{\bs{N}} & =& 4\int_0^\cdot I(\bs{r}_u>0,\bs{s}_u>0)\,du \qquad\mathbb{P}\mbox{-almost surely.}
\end{eqnarray}
The proof does not involve new arguments, so we just sketch the main steps:
With the same arguments as before,
one can conclude that
$$t\mapsto \bs{M}_t\bs{N}_t-4\int_0^t I(\bs{r}_u>0,\bs{s}_u>0)\,du$$ is a submartingale and that the map
$t\mapsto \bs{M}_t\bs{N}_t-4t\, $ is $\mathbb{P}$-almost surely decreasing.
Moreover, by \eqref{eq:qvM}, \eqref{eq:qvN}, and the Kunita-Watanabe inequality, we see that $\mathbb P${-a.s.},
$$\int_s^t I(\bs{r}_u=0 \text{ or } \bs{s}_u=0)\,d\coV{\bs{M}}{\bs{N}}_u\ =\ 0
\qquad\mbox{for }0\le s\le t,\qquad\mbox{and thus}
$$
\begin{eqnarray*}
	\coV{\bs{M}}{\bs{N}}_t-\coV{\bs{M}}{\bs{N}}_s& =& \int_s^t I(\bs{r}_u>0,\bs{s}_u>0)\,d\coV{\bs{M}}{\bs{N}}_u\\ & \le &  4 \int_s^t I(\bs{r}_u>0,\bs{s}_u>0)\,du \qquad\mbox{for }0\le s\le t.
\end{eqnarray*}
This completes the proof of \eqref{c1}.
Invoking a martingale representation theorem, see e.g.\ \cite[Ch.\ II, Theorem 7.1']{MR1011252}, we conclude that
there is a probability space $(\Omega, \mathcal{A}, P)$ supporting a Brownian motion $\tilde{W}$, and random variables
$(\tilde{r},\tilde{s})$ such that $P\circ (\tilde{r},\tilde{s})^{-1}=\mathbb{P}\circ(\bs{r},\bs{s})^{-1}$,  and such that
$(\tilde{r}_t,\tilde{s}_t,\tilde{W}_t)$ is a weak solution of \eqref{eq:thmApproximation}.\smallskip

It remains to show that \eqref{m1} and \eqref{m2} imply $P[\tilde{r}_t \leq \tilde{s}_t \text{ for all } t\geq 0]=1$.
Applying a comparison theorem \cite[Theorem 1]{MR0471082} to the approximating diffusions \eqref{eq:stickApprox1} shows that
$\mathbb{P}^n[\bs{r}_t\leq \bs{s}_t \text{ for all } t\geq 0]=1$ for all $n$. The monotonicity carries over to the limit, since
$\mathbb{P}^n\circ(\bs{r},\bs{s})^{-1}$ converges weakly, along a subsequence, towards $\mathbb{P}\circ(\bs{r},\bs{s})^{-1}$.
\end{proof}

\subsection{Long time behavior}
We now derive bounds for solutions to \eqref{eq:oneDimSticky} that are stable for long times. We assume that $t\mapsto\alpha(t,x)$ is non-increasing, so that the \emph{stickiness} of solutions
to \eqref{eq:oneDimSticky}
is non-decreasing in time.
\begin{ass}\label{ass:summaryTimeDep}
	The function $\alpha:[0,\infty )\times [0,\infty ) \rightarrow\reals$ is
	locally Lipschitz continuous with $\alpha(t,x)\leq \alpha(s,x)$ for any $s\leq t$ and $x\in\reals_+$, $\alpha(t,0)>0$ for any
	$t\ge 0$,  and
	\begin{eqnarray}\label{ass:convexityAtInfinity2}
			\limsup_{r\rightarrow\infty}\,  (r^{-1}{\alpha(0,r)}) < 0.
	\end{eqnarray}
\end{ass}
Notice that Assumption \ref{ass:summaryTimeDep} implies Assumptions  \ref{ass:positiveAtZero}, \ref{ass:stickySdeDriftLocallyLipschitzAndPositive} and \ref{ass:stickySdeNonExplosiveCondition} from above.

\subsubsection{Invariant measure in the time-homogenous case}
We first consider
drift coefficients which do not depend on time, i.e., functions of the form $\alpha(t,x)=\alpha(x)$.

\begin{lem}\label{invariantmeasure}
Suppose that Assumption 
\ref{ass:summaryTimeDep} holds true, and $\alpha (t,\cdot )=\alpha$ for a  function $\alpha:[0,\infty )\rightarrow\reals$. Let $\pi$ be the probability measure on
$[0,\infty )$ defined by
	\begin{eqnarray}\label{eq:pi}
	 \pi(dx) &=& \frac 1Z\, \left(\frac{2}{\alpha(0)}\ \delta_0(dx) \ +\ \exp\left(\frac 12\int_0^x \alpha(y)\, dy\right)\lambda_{(0,\infty)}(dx)\right) \end{eqnarray}
where $Z = \frac{2}{\alpha(0)} + \int_0^\infty \exp\left(\frac 12\int_0^x\alpha(y)\, dy\right) dx $. Then $\pi$ is invariant for 
	\eqref{eq:oneDimSticky}, i.e., if $(r_t)$ is a solution
	with initial law $\pi$, then  $\operatorname{Law}(r_t)=\pi$ for any $t\geq 0$.
\end{lem}

\begin{proof}
We use an approximation as in \eqref{eq:stickApprox1} with $\beta (t,x)=\alpha (x)$ and 
a sequence of smooth functions $\vartheta^n:[0,\infty )\to [0,1]$
satisfying Assumption \ref{ass:DiffusionCoefficientApproximation}. 
It is well-known that under our assumptions, for each $n\in\mathbb N$, the probability measure $\tilde{\mu}^n$ on $\mathbb R_+$ with
distribution function
	\begin{eqnarray*}
		\tilde{F}^n(x) =
		\frac{\int_0^x \frac{1}{\vartheta^n(y)^2}\exp\left(\int_{1/n}^y \frac{\alpha(z)}{2\vartheta^n(z)^2} dz\right) dy}
		{\int_0^\infty \frac{1}{\vartheta^n(y)^2} \exp\left(\int_{1/n}^y \frac{\alpha(z)}{2\vartheta^n(z)^2} dz\right) dy} \qquad x\in [0,\infty ),
	\end{eqnarray*}
	is an invariant measure for the process $(\tilde r_t^n)$  
	defined by \eqref{eq:stickApprox1}, see e.g.\ \cite[Chapter 4.4,
	Theorem 7]{MR0247667}.
	Note in particular that by Assumptions \ref{ass:summaryTimeDep} and
	\ref{ass:DiffusionCoefficientApproximation}, the occurring integrals
	are well defined and finite. Let $F$ denote the distribution function of
	$\pi$. We show that for any $x>0$,
	$\tilde{F}^n(x)\rightarrow F(x)$ as $n\rightarrow\infty$, 
	which implies that $\tilde\mu_n\rightarrow \pi$ weakly. Indeed, 
	fix $x\in (0,\infty ]$. Then for $n>1/x$,
	\begin{eqnarray}\label{eq112}
		&&\int_0^x \frac{1}{\vartheta^n(y)^2}\exp\left(\int_{1/n}^y \frac{\alpha(z)}{2\vartheta^n(z)^2} dz\right) dy \\\nonumber
		&=& 	\int_{1/n}^x \exp\left(\int_{1/n}^y \frac 12\,\alpha(z)  dz\right) dy
		+\int_{0}^{1/n} \frac{1}{\vartheta^n(y)^2}\exp\left(\int_{1/n}^y \frac{\alpha(z)}{2\vartheta^n(z)^2} dz\right) dy.
	\end{eqnarray}
	If $C\in(0,\infty)$ is a constant then
	\begin{align}\nonumber
	\int_{0}^{1/n} \frac{1}{\vartheta^n(y)^2}\exp\left(\int_{1/n}^y \frac{C}{\vartheta^n(z)^2} dz\right) dy
		\, =\,  \lim_{\epsilon\downarrow 0} 		
		 \int_{\epsilon}^{1/n} \frac{1}{\vartheta^n(y)^2}\exp\left(\int_{1/n}^y 
		 \frac{C}{\vartheta^n(z)^2} dz\right) dy\end{align}
\begin{eqnarray}
\ =\	 \lim_{\epsilon\downarrow 0} \frac{1}{C} \left(1-\exp\left(-\int^{1/n}_\epsilon \frac{C}{\vartheta^n(z)^2} dz\right)\right)
		\ =\  \frac{1}{C}.\label{eq111}
\end{eqnarray}	
	For $0<y<1/n$, we have the bounds
	\begin{eqnarray*}
		\exp\left(\max_{u\in[0,1/n]}\alpha(u)\int_{1/n}^y \frac{1}{2\vartheta^n(z)^2} dz\right) &\leq&
		\exp\left(\int_{1/n}^y \frac{\alpha(z)}{2\vartheta^n(z)^2} dz\right)
		\\&\leq&  \exp\left(\min_{u\in[0,1/n]}\alpha(u)\int_{1/n}^y \frac{1}{2\vartheta^n(z)^2} dz\right).
	\end{eqnarray*}
	Using \eqref{eq112}, the continuity of $\alpha$, and \eqref{eq111}, we can conclude that as $n\rightarrow\infty$,
	\begin{eqnarray*}
		\int_0^x \frac{1}{\vartheta^n(y)^2}\exp\left(\int_{1/n}^y \frac{\alpha(z)}{2\vartheta^n(z)^2} dz\right) dy
		&\rightarrow& 	\int_{0}^x \exp\left(\int_{0}^y \frac 12\,\alpha(z)  dz\right) dy
		+\frac{2}{\alpha(0)}.
	\end{eqnarray*}
	Since this also holds for $x=\infty$, we see that $\tilde{F}^n(x)\rightarrow F(x)$ for any $x>0$, and hence $\tilde\mu_n\rightarrow \pi$ weakly. 
	Consequently,  by Lemma \ref{lem:uniqueness}
	and Theorem \ref{thmApproximation}, the laws of the solutions of \eqref{eq:stickApprox1} with initial distributions $\tilde{\mu}^n$ converge
	weakly to the law of the
	solution of \eqref{eq:oneDimSticky} with initial distribution $\pi$. 
  Since the approximating processes are stationary, the limit process is
  stationary, too. Hence $\pi$ is an invariant measure.	
\end{proof}

\subsubsection{Long time stability in the time-inhomogeneous case} 
 
Let $(r_t)$ be a solution of \eqref{eq:oneDimSticky} with an arbitrary but fixed initial distribution $\mu$ on $\reals_+$.
Our aim is to provide bounds on $P[r_t>0]$ and $E[r_t]$ for any fixed $t\geq 0$. To this end we fix a continuous function $a:[0,\infty)\to\reals$
such that 
	\begin{eqnarray}\label{ass:convexityAtInfinity}
	\alpha (0,x)& \le &  a(x)\,\mbox{ for any }x\in [0,\infty ),\mbox{ \ and \ }	
			\limsup_{r\rightarrow\infty}\,  (r^{-1}{a(r)}) \ <\ 0.
	\end{eqnarray}
For example, by Assumption \ref{ass:summaryTimeDep}, we can always choose $a(x)=\alpha (0,x)$. However, sometimes it can be more convenient
	to choose the function $a$ in a different way.
Following \cite{MR2843007,Eberle2015} (see also \cite{MR972776,MR1345035,MR1401516,MR2152380}), we
define constants $R_0,R_1\in(0,\infty)$ and a concave function
$f:\reals_+\rightarrow\reals_+$ by
\begin{eqnarray}\label{R0}
	R_0 &=& \inf\{ R\geq 0 : a(r)\leq 0\quad\text{for any } r\geq R \},\\\label{R1}
	R_1 &=& \inf\{ R\geq R_0 : R(R-R_0)\;a(r)/r\leq -4 \quad\text{for any } r\geq R  \},\\
\label{f}
	f(r)&=&\int_0^r \phi(s)\ g(s)\ ds,\, \mbox{ where }\, \phi(r) \ =\ \exp\left(-\frac{1}{2} \int_0^r a(s)^+ ds\right)\ \mbox{and}
\end{eqnarray}
\begin{eqnarray*}
g(r) &=& 1- \frac{1}{4} \int_0^{r\wedge R_1} \frac{\Phi(s)}{\phi(s)}\, ds \left/ \int_0^{R_1} \frac{\Phi(s)}{\phi(s)}\, ds\right.
		- \frac{1}{4} \int_0^{r\wedge R_1} \frac{1}{\phi(s)}\, ds \left/ \int_0^{R_1} \frac{1}{\phi(s)}\, ds\right.
\end{eqnarray*}
with
$ \Phi(r) =  \int_0^r \phi(s) \, ds$. The function $f$ is concave, strictly increasing and continuous.
Observe that \eqref{ass:convexityAtInfinity} implies that
$0<R_0<R_1<\infty$. We define constants
\begin{align}\label{ceps}
	c \ =\ \left(2{\int_0^{R_1} \frac{\Phi(s)}{\phi(s)}\ ds}\right)^{-1}, \quad \epsilon \
	= \ \min\left\{\left(2\int_0^{R_1} \frac{1}{\phi(s)}\ ds\right)^{-1}, \ c\,\Phi(R_1) \right\}.
\end{align}
Notice that $1/2\leq g\leq 1$, and thus $\Phi(r)/2\leq f(r)\leq \Phi(r)$.
Hence for $0<r<R_1$,
\begin{eqnarray}\label{eqf}
	2\,f''(r)\ + \ f'(r)\,a(r)^+ &\le & -\epsilon \, -\, c\ \Phi(r)
	\ \leq\  -\left(\epsilon \, +\,  c\ f(r)\right).
\end{eqnarray}

\begin{lem}\label{eq:ergodicestimatesuntilT0}
Suppose that Assumption \ref{ass:summaryTimeDep}  holds true. Let $(r_t)$ be a solution of \eqref{eq:oneDimSticky}, 
and let $T_0=\inf\{t\geq 0: r_t=0\}$. Then for any $t>0$,
\begin{eqnarray}
	E[f(r_t)\,;\, t<T_0] &\leq & e^{-c\,t}\ E[f(r_0)], \qquad\qquad\text{and}\\
		P[t<T_0] &\leq & \frac{1}{\epsilon} \ \frac{c}{e^{c\,t}-1} \ E[f(r_0)] .
\end{eqnarray}
\end{lem}

\begin{proof}
	Notice that the function $f$ can be extended to a concave function on
	$\reals$ by setting $f(x)=x$ for $x<0$. Since the process $(r_t)$ is a
	continuous semimartingale, we can apply the It{\^o}-Tanaka formula to conclude
	that almost surely,
\begin{eqnarray}\label{p1}
	df(r_t) &=& f'(r_t)\, \alpha(t,r_t) \ dt + 2\, f''(r_t)\, I(r_t>0) \ dt +  dM_t,
\end{eqnarray}
where $M_t=2 \int_0^t f'(r_s)\,I(r_s>0)\  dW_s$ is a martingale.
By  Assumption \ref{ass:summaryTimeDep} and \eqref{ass:convexityAtInfinity}, $\alpha (t,r_t)\le \alpha (0,r_t)\le a(r_t)$. Therefore, for
$0<r_t<R_1$, we can apply \eqref{eqf} to bound the right hand side of \eqref{p1}. On the other hand,
for $r_t\geq R_1$, we have $f''(r_t)=0$ and
$r_t^{-1}\alpha(r_t) < 0$. Moreover, by definition of $f$ and $\phi$, $f'(r_t)=\phi(R_0)/2$, and by
\eqref{R1}, $R_1(R_1-R_0) \alpha(r_t)/{r_t}^{-1}\leq -4$.
Therefore, we can conclude similarly to \cite[Proof of Theorem 2.2]{Eberle2015} that for $r_t>R_1$,
\begin{eqnarray}\nonumber
	f'(r_t)\, \alpha(t,r_t) &\leq& \phi(R_0) a(r_t) /2
	\,\leq\, -2 \frac{\phi(R_0)}{R_1-R_0} \frac{r_t}{R_1} \,<\, -2
	\frac{\phi(R_0)}{R_1-R_0} \frac{\Phi(r_t)}{\Phi(R_1)} \\&\leq& 
	- \,\Phi(r_t) \left/ \int_{R_0}^{R_1} \Phi(s)\phi(s)^{-1}\, ds\right. \leq\ - 2\,c\, \Phi(r_t)\label{p2} \\ &\leq &  -c\,\Phi(R_1)  \ -\ c\,f(r_t)
	\ \leq \  -\left(\epsilon +\  c\, f(r_t) \right).  \nonumber
\end{eqnarray}
Here we have used that $\int_{R_0}^{R_1} \Phi(s)\phi(s)^{-1}\,ds\geq
(R_1-R_0)\Phi(R_1)\phi(R_0)^{-1}/2$.
Combining \eqref{p1}, \eqref{eqf} and \eqref{p2}, we see that  almost surely,
\begin{equation}
\label{p3}
	df(r_t)\ \le\ -\left(\epsilon +\  c\, f(r_t) \right) \ dt \ + \ dM_t\qquad\mbox{for }t<T_0.
\end{equation}
Using It\^o's product
rule and \eqref{p3}, we finally obtain
\begin{eqnarray*}
e^{ct}\, E[f(r_t);\, t<T_0] &\leq& E[f(r_0)]\ +\ E[e^{c (t\wedge T_0)} f(r_{t\wedge T_0})-f(r_0)]   \\
 &\leq& E[f(r_0)]\ -\ \frac{\epsilon}{c}\ \left(E\left[e^{c\,(t\wedge T_0)}\right]-1\right) , \quad\mbox{and}\\
 P[t<T_0] &\leq&  E\left[\frac{e^{c\,(t\wedge T_0)}-1}{e^{c\,t}-1}\right]\
		\ \leq \  \frac{1}{\epsilon} \ \frac{c}{e^{c\,t}-1} \ E[f(r_0)].
\end{eqnarray*}
\end{proof}

For $s\in [0,\infty )$, we denote by $\pi_s$ the invariant probability measure
for the {\em time-homogeneous} sticky diffusion with drift $\alpha (s,\cdot)$ that is given by \eqref{eq:pi}, i.e.,
 \begin{equation}\label{eq:pis}
	 \pi_{s}(dx) \ \propto\ \frac{2}{\alpha(s,0)}\ \delta_0(dx) \
	 +\ \exp\left(\frac{1}{2}\int_0^x \alpha(s,y)\,
	 dy\right)\lambda_{(0,\infty)}(dx).
\end{equation}
\begin{thm}\label{quantification1}
Suppose that Assumption \ref{ass:summaryTimeDep}  holds true, and let $(r_t)$ be a solution of \eqref{eq:oneDimSticky}
with initial distribution $\mu$ on $\reals_+$. 
Then
for any $t>0$,
\begin{eqnarray*}
		E[f(r_t)] &\leq & e^{-c\,t}\ E[f(r_0)] \ + \ \int f \, d\pi_0,\quad
		E[r_t] \ \leq\ 2\, \phi(R_0)^{-1} E[f(r_t)],\quad\mbox{and}\\
		P[r_t>0]&\leq & \frac{1}{\epsilon} \ \frac{c}{e^{c\,t}-1} \ E[f(r_0)] \ + \ \pi_0[(0,\infty)].
\end{eqnarray*}
\end{thm}

\begin{proof}
Based on the results of Theorem \ref{thmApproximation}, we can construct a
filtered probability space $(\Omega, \mathcal{A}, (\filt_t), P)$
satisfying the usual conditions and supporting random variables 
$r,W,\tilde{r},\tilde{s},\tilde{W}:\Omega\rightarrow\mathbb{W}$ such that
w.r.t.\ $(\Omega, \mathcal{A}, (\filt_t), P)$,
\begin{itemize}
\item  $(r,W)$ and $(\tilde{r},\tilde{s},\tilde{W})$ are independent,
  \item $(r_t,W_t)$ is a weak solution of  \eqref{eq:oneDimSticky} with initial
  distribution $\mu$, \ and
  \item $(\tilde{r}_t,\tilde{s}_t,\tilde{W}_t)$ is a weak solution of \eqref{eq:thmApproximation}
with $\beta(t,x)=\alpha(t,x)$,
\-$\gamma(t,x)=\alpha(0,x)$, $\tilde{\mu}=\delta_0$, $\tilde{\nu}=\pi_0$, and
\begin{equation}
\label{eq:compare}
P[\ \tilde{r}_t \leq \tilde{s}_t \text{ for all }t\geq 0\ ]\ =\ 1.
\end{equation}
\end{itemize}
Let $T:=\inf\{t\geq 0: r_t=\tilde{r}_t\}$ be the first meeting time of $(r_t)$ and $(\tilde{r}_t)$.
 We define
 \begin{eqnarray}\label{coupling}
	\bar{r}_t &:=& r_t \quad \text{for } t<T,\quad\text{and}\quad \bar{r}_t \ := \ \tilde{r}_{t}  \quad \text{for } t\geq T.
\end{eqnarray}
Then $(\bar r_t)$ solves the martingale problem corresponding to \eqref{eq:oneDimSticky} with initial law $\mu$, cf.\ e.g.\ \cite[Section 3.1]{MR2227134}. By
Lemma \ref{lem:uniqueness}, this martingale problem has a unique solution.
Hence, we can conclude that the laws of $\bar{r}$ and $r$ on $\mathbb{W}$ coincide.
Let $T_0=\inf\{t\geq 0: r_t=0\}$. Observe that since 
$t\mapsto r_t$ and $t\mapsto\tilde{r}_t$ are continuous with $\tilde{r}_0=0 \le r_0$, we have $T\leq T_0$.
In particular, by Lemma \ref{eq:ergodicestimatesuntilT0},
\eqref{eq:compare}, and since $f$ is increasing, 
\begin{eqnarray*}
	E[f(r_t)] &=&  E[f(\bar{r}_t)] \ = \ E[f(r_t); \, t<T] \ + \  E[f(\tilde{r}_t);\, t\geq T]
	\\&\leq& E[f(r_t); \, t<T_0] \, + \, E[f(\tilde{s}_t)]
	\ \leq\  e^{-c\,t}\, E[f(r_0)] \, + \, \int f \, d\pi_0.
\end{eqnarray*}
Here we have used that by Lemma \ref{invariantmeasure}, the process
$(\tilde s_t)$ is stationary. By  \eqref{f}, \eqref{R0}, and since $g\ge 1/2$,
we have $f'\ge \phi (R_0)/2$. Hence
the inequality $r \leq 2\,\phi(R_0)^{-1} f(r)$
holds for any $r\geq 0$, and thus, we can conclude that $$E[r_t] \ \leq\
2\phi(R_0)^{-1} E[f(r_t)].$$ 
Finally, by the second part of Lemma
\ref{eq:ergodicestimatesuntilT0}, we see that
\begin{eqnarray*}
	P[r_t>0] &=& P[\bar{r}_t>0] \ = \  P[r_t>0,\, t<T] + P[\tilde{r}_t>0,\, t\geq T]
	\\&\leq& P[t<T_0] \ +\  P[\tilde{s}_t>0]  \ \leq \
	\frac{1}{\epsilon} \ \frac{c}{e^{c\,t}-1} \ E[f(r_0)] \ + \ \pi_0[(0,\infty)].
\end{eqnarray*}
\end{proof}

By applying Theorem \ref{quantification1} on the time intervals $[s,t]$ and
$[0,s]$, we obtain:

\begin{cor}\label{thmtimedependent}
Suppose that Assumption \ref{ass:summaryTimeDep}  holds true, and let $(r_t)$ be a solution of \eqref{eq:oneDimSticky}. 
Then
for any $0\le s< t$,
\begin{eqnarray*}
		E[f(r_t)] & \leq&   e^{-ct}E[f(r_0)] \,+\, e^{-c\, (t-s)} \int f\,d\pi_0\, +\,  \int f \,d\pi_{s},\qquad\mbox{and}\\
		P[r_t>0]&\leq & \frac{1}{\epsilon} \ \frac{c}{e^{c\,(t-s)}-1} \,\left( e^{-cs}E[f(r_0)] + \int f\,d\pi_0\right) \, + \, \pi_{s}[(0,\infty)].
\end{eqnarray*}
where $f$, $c$ and $\epsilon$ are defined as above. Furthermore,
\begin{eqnarray*}
		E[f_s(r_t)] & \leq& \frac{2}{\phi (R_0)}\, e^{-c_s (t-s)} \left(  e^{-cs}E[f(r_0)] +  \int f\,d\pi_0\right)\, +\,  \int f _s\,d\pi_{s},\qquad\mbox{and}\\
		P[r_t>0]&\leq &\frac{2}{\phi (R_0)\epsilon_s}\,  \ \frac{c_s}{e^{c_s\,(t-s)}-1} \,\left( e^{-cs}E[f(r_0)] + \int f\,d\pi_0\right) \, + \, \pi_{s}[(0,\infty)],
\end{eqnarray*}
where $f_s$, $c_s$ and $\epsilon_s$ are defined by \eqref{f}, \eqref{ceps}  and \eqref{eq:pi} with $a$ replaced by $\alpha (s,\cdot)$.
\end{cor}

\begin{proof}
Fix $s\in [0,\infty )$. Then the process $(r_{s+t})_{t\ge 0}$ solves \eqref{eq:oneDimSticky} with drift coefficient $\alpha_s(t,x)=\alpha (s+t,x)$
and initial distribution $P\circ r_s^{-1}$. Since $\alpha_s(t,x)\le \alpha (s,x)
\le a(x)$ for any $t,x\ge 0$, we can apply Theorem \ref{quantification1}
either with $a,f,c$ and $\epsilon$ as above, or with  $a,f,c$ and $\epsilon$
replaced by $\alpha (s,\cdot ),f_s,c_s$ and $\epsilon_s$. For $t>s$ we obtain
\begin{eqnarray*}
		E[f(r_t)] &\leq & e^{-c\,(t-s)}\ E[f(r_s)] \ + \ \int f \, d\pi_s,\\
		P[r_t>0]&\leq & \frac{1}{\epsilon} \ \frac{c}{e^{c\,(t-s)}-1} \ E[f(r_s)] \ + \ \pi_s[(0,\infty)],\\
		E[f_s(r_t)] &\leq & e^{-c_s\,(t-s)}\ E[f_s(r_s)] \ + \ \int f_s \, d\pi_s,\\
		P[r_t>0]&\leq & \frac{1}{\epsilon_s} \ \frac{c_s}{e^{c_s\,(t-s)}-1} \ E[f_s(r_s)] \ + \ \pi_s[(0,\infty)].
\end{eqnarray*}
Noting that $f_s(r_s)\le r_s$, the assertion follows by 
applying Theorem \ref{quantification1} once more.
\end{proof}

\section{Coupling construction and proofs of the main results} \label{sec:proofs}
In this section, we prove our main theorems. First of all, we construct the
sticky coupling $(X_t,Y_t)$ of solutions to \eqref{eq:DiffusionDifferentDrifts}
and \eqref{eq:DiffusionDifferentDrifts2} respectively, advertised in
Theorem \ref{main1}. The coupling is obtained as a weak limit of
Markovian couplings $(X^\delta_t,Y_t^\delta),\ {\delta>0}$. The couplings
$(X^\delta_t,Y_t^\delta)$  are reflection couplings for
$|X_t^\delta-Y^\delta_t|\geq \delta$ and synchronous couplings for
$|X_t^\delta-Y^\delta_t|=0$. Inbetween there is an interpolation between the two types of
couplings. We argue that the family of couplings is tight and
thus there is a subsequence converging to a coupling $(X_t,Y_t)_{t\ge 0}$. 
It is then argued that this limiting coupling is sticky and shares the
properties stated in Theorem \ref{main1}.
\par\medskip
We now define the couplings $(X^\delta_t,Y^\delta_t)$ 
rigorously. The technical realization follows
\cite{Eberle2015}.
We introduce Lipschitz functions
$\operatorname{rc}^\delta,\operatorname{sc}^\delta:\reals_+\rightarrow[0,1]$ such that
$\operatorname{rc}^\delta(0) = 0$, $ \operatorname{rc}^\delta(r)  > 0$ for 
$ 0<r< \delta$,
$\operatorname{rc}^\delta(r)  =  1$ for $ r\geq \delta$, and
   \begin{eqnarray}\label{r2+s2=1}
  \operatorname{rc}^\delta(r)^2+\operatorname{sc}^\delta(r)^2 & = & 1 \quad  \text{ for any } r\geq 0.
  \end{eqnarray}
  Let
  $(B_t^1)$ and $(B_t^2)$ be independent $d$-dimensional Brownian motions, and let $u\in\reals^d$ be some arbitrary
   unit vector. We define the coupling $(X_t^\delta,Y_t^\delta)$ for
   \eqref{eq:DiffusionDifferentDrifts} and \eqref{eq:DiffusionDifferentDrifts2}
   as a diffusion process in $\reals^{2d}$ satisfying the stochastic differential equation
\begin{alignat}{4}\label{coup1}
  dX_t^\delta &\ =\  b(t,X^\delta_t)\ dt \ +&\ \operatorname{rc}^\delta\left(\tilde{r}^\delta_t\right)\ dB^1_t
  	&\ + \ \operatorname{sc}^\delta\left(\tilde{r}^\delta_t\right)\ dB^2_t,
  	\\\label{coup2}
  	dY_t^\delta &\ = \  \tilde{b}(t,Y^\delta_t)\ dt \ +& \ 
  	\operatorname{rc}^\delta\left(\tilde{r}^\delta_t\right)\ \left(\operatorname{Id}_{\reals^d}-2\,e^\delta_t\sProd{e^\delta_t}{\cdot}\right)\ dB^1_t
  	 	&\ +\ \operatorname{sc}^\delta\left(\tilde{r}^\delta_t\right)\ dB^2_t,
  \end{alignat}
  with initial condition $(X^\delta_0,Y^\delta_0) =(x,y)$. Here $Z^\delta_t=X_t^\delta-Y_t^\delta$,
  $\tilde{r}^\delta_t=\norm{Z_t^\delta}$, $e^\delta_t=
  {Z^\delta_t}/{\tilde{r}^\delta_t}$ if $\tilde{r}^\delta_t\not=0$, and $e^\delta_t=u$ if
  $\tilde{r}^\delta_t=0$.
  		Since $\operatorname{rc}^\delta(0)=0$, the arbitrary value $u$ is not relevant for the dynamics.
 The process $(X_t^\delta,Y_t^\delta)$ can be realized as a standard diffusion process in $\reals^{2d}$ 
 with locally Lipschitz coefficients. Moreover, Assumptions
 \ref{ass:BoundOnDriftDifference} and \ref{ass:CurvControlForB} imply the
 non-explosiveness of the process. Using Lévy's characterization of Brownian
 motion and \eqref{r2+s2=1}, one can check that  $(X_t^\delta,Y_t^\delta)$ is indeed a coupling of solutions to Equations \eqref{eq:DiffusionDifferentDrifts} and \eqref{eq:DiffusionDifferentDrifts2}. Notice that the process $W^\delta_t=\int_0^t \sProd{e^\delta_s}{dB^1_s}$ is a one-dimensional Brownian motion.

\begin{lem}\label{lem7}
Suppose that Assumptions \ref{ass:BoundOnDriftDifference} and
\ref{ass:CurvControlForB} are satisfied. Then, almost surely,
  		\begin{alignat}{2}\label{lem7e1}
  	d\tilde{r}^\delta_t \ &= \  &\sProd{e^\delta_t}{b(t,X^\delta_t)-\tilde{b}(t,Y^\delta_t)}\ dt \ &+\ 2\, \operatorname{rc}^\delta\left(\tilde{r}^\delta_t\right) \ dW^\delta_t\phantom{.}
  	\\\label{lem7e2} &\leq \  &\left(\  M\ +\ \kappa(\tilde{r}^\delta_t)\,
  	\tilde{r}^\delta_t \ \right)\ dt \ &+\ 2\, \operatorname{rc}^\delta\left(\tilde{r}^\delta_t\right) \ dW^\delta_t.
  \end{alignat}
\end{lem}
  
\begin{proof}
By \eqref{coup1} and \eqref{coup2},
	\begin{align*}
		d(\tilde{r}^\delta_t)^2 =
	2 \sProd{Z^\delta_t}{b(t,X^\delta_t)-\tilde{b}(t,Y^\delta_t)} dt + 4
	\operatorname{rc}^\delta\left(\tilde{r}^\delta_t\right)^2 dt + 4
	\operatorname{rc}^\delta\left(\tilde{r}^\delta_t\right) \sProd{Z^\delta_t}{e_t^\delta}
	 dW^\delta_t.
	\end{align*}
	For $\epsilon>0$, we define a $C^2$
	approximation of the square root by
		\begin{equation}
 			S_\epsilon(r)\ =\ 
 		  -({1}/{8})\;\epsilon^{-3/2}\; r^2\, +\, ({3}/{4})\; \epsilon^{-1/2}\; r\,
 		  +\, ({3}/{8})\;\epsilon^{1/2} \qquad \mbox{ for } r<\epsilon,
 	\end{equation}
$S_\epsilon (r)=  \sqrt{r}	$ for $ r\geq \epsilon$. By It\^o's formula, 
 	\begin{align*}
 		dS_\epsilon((\tilde{r}^\delta_t)^2) &= 2\,S'_\epsilon((\tilde{r}^\delta_t)^2)
 		\sProd{Z^\delta_t}{b(t,X^\delta_t)-\tilde{b}(t,Y^\delta_t)}\ dt +
 		4 S'_\epsilon((\tilde{r}^\delta_t)^2) 
	\operatorname{rc}^\delta\left(\tilde{r}^\delta_t\right)^2 dt \\&+ 8\,S''_\epsilon((\tilde{r}^\delta_t)^2)
 		\operatorname{rc}^\delta\left(\tilde{r}^\delta_t\right)^2
 		(\tilde{r}^\delta_t)^2\,dt \,+\,
 		4\,S'_\epsilon((\tilde{r}^\delta_t)^2)\,
	\operatorname{rc}^\delta\left(\tilde{r}^\delta_t\right) \tilde{r}^\delta_t
	\, dW^\delta_t.
 	\end{align*}
	We can now pass to the limit $\epsilon\downarrow 0$ to obtain \eqref{lem7e1}.
	Notice that $\sup_{0\leq r\leq\epsilon} \norm{S'_\epsilon(r)}\lesssim
	\epsilon^{-1/2}$, $\sup_{0\leq r\leq\epsilon} \norm{S''_\epsilon(r)}\lesssim
	\epsilon^{-3/2}$ and that $\operatorname{rc}^\delta$ is Lipschitz with
	$\operatorname{rc}^\delta(0)=0$.
	Hence, one can use
	Lebesgue's dominated convergence theorem for the convergence of the first
	three integrals. Moreover, the 
	stochastic integral converges almost surely, along a subsequence, to $\int_0^t
	2\, \operatorname{rc}^\delta\left(\tilde{r}^\delta_s\right) \, dW^\delta_s$.	
Finally, by Assumptions \ref{ass:BoundOnDriftDifference} and
\ref{ass:CurvControlForB},
\begin{eqnarray*}
	\sProd{Z^\delta_t}{b(t,X^\delta_t)-\tilde{b}(t,Y^\delta_t)} &\leq& \sProd{Z^\delta_t}{b(t,X^\delta_t)-b(t,Y_t^\delta)+b(t,Y_t^\delta)-\tilde{b}(t,Y^\delta_t)}
	\\&\leq& M \ \tilde{r}^\delta_t \ +\  \kappa(\tilde{r}^\delta_t) \
	(\tilde{r}^\delta_t)^2.
\end{eqnarray*}
\end{proof}

In order to control the distance of $X^\delta_t$ and $Y^\delta_t$, we introduce a one-dimensional process $(r^\delta_t)$ that is defined as the unique and
strong solution to the equation
\begin{eqnarray} \label{eq:oneDimContr}
 	dr_t^\delta &=& \left(\  M\ +\ \kappa(r^\delta_t)\cdot r^\delta_t\ \right)\ dt
 	\ +\ 2\, \operatorname{rc}^\delta\left(r^\delta_t\right) \ dW^\delta_t, \quad r_0^\delta=\tilde{r}^\delta_0,
\end{eqnarray} 
 with $(\tilde{r}^\delta_t)$ and $(W^\delta_t)$ as above. 
  
 \begin{lem}\label{lemmon1} We have 
 	 $\norm{X_t^\delta-Y_t^\delta}= \tilde{r}^\delta_t \leq r_t^\delta$, almost
 	 surely for all $t\geq 0$.
 \end{lem}
 \begin{proof}
 The processes $(\tilde{r}^\delta_t)$
 	and $(r^\delta_t)$ are driven by the same noise, start at the same position,
and, by \eqref{lem7e2}, the drift of $(\tilde{r}^\delta_t)$ is smaller
 	or equal to the one of $(r^\delta_t)$.
Therefore, the assertion follows by Ikeda-Watanabe's comparison theorem for one-di\-mensional
 	diffusions, cf.\
 	\cite[Theorem 1.1]{MR0471082}.
 \end{proof}

\begin{proof}[Proof of Theorem \ref{main1}]
We consider the diffusion $U^\delta_t:=(X_t^\delta,Y_t^\delta,r_t^\delta)$ on $\reals^{2d+1}$. 
Let $\mathbb{P}^\delta$ denote the law of $U^\delta$ on the space ${C}(\reals_+,\reals^{2d+1})$. 
We define
$\boldsymbol{X},
\boldsymbol{Y}:{C}(\reals_+,\reals^{2d+1})\to {C}(\reals_+,\reals^{d})$
and
$\boldsymbol{r}:{C}(\reals_+,\reals^{2d+1})\to {C}(\reals_+,\reals)$ as the canonical projections onto the first $d$, the second $d$, and the last coordinate.

Notice that
in each of the equations \eqref{coup1}, \eqref{coup2} and
\eqref{eq:oneDimContr}, the drift coefficients do not depend on $\delta$ and the diffusion coefficients
are uniformly bounded.  Moreover, Assumptions \ref{ass:BoundOnDriftDifference}
and \ref{ass:CurvControlForB} imply that, similarly as in the proof of Theorem \ref{thmApproximation}, the diffusions
 $(U_t^\delta)$ satisfy uniformly a Lyapunov non-explosion criterion,  
  and the
 drift coefficients are uniformly bounded on compact sets. Therefore, the family 
 $(\mathbb{P}^\delta)$ is tight, cf.\ \cite{MR2870529,MR3175790}.
In particular, there is a sequence $\delta_n\downarrow 0$ such that
$(\mathbb{P}^{\delta_n})$ converges towards a measure $\mathbb{P}$
on $\mathbb{C}(\reals_+,\reals^{2d+1})$.
For each $\delta>0$, $(X^\delta_t)$ and $(Y^\delta_t)$ are solutions to 
\eqref{eq:DiffusionDifferentDrifts} and \eqref{eq:DiffusionDifferentDrifts2}
respectively. Since those solutions are unique in law, we know that
$\mathbb{P}^\delta\circ (X^\delta)^{-1} = \mathbb{P}\circ \boldsymbol{X}^{-1}$ and $\mathbb{P}^\delta\circ (Y^\delta)^{-1} = \mathbb{P}\circ \boldsymbol{Y}^{-1}$
	for any $\delta>0$. Hence, $ \mathbb{P}\circ (\boldsymbol{X},
	\boldsymbol{Y})^{-1}$ is a coupling of \eqref{eq:DiffusionDifferentDrifts} and
	\eqref{eq:DiffusionDifferentDrifts2}. 	
	Moreover, Lemma \ref{lem:uniqueness} and the proof of Theorem \ref{thmApproximation} reveal that, 
	after extending
	the underlying probability space, there is a Brownian motion $(\tilde{W}_t)$ such that  
	$(r_t,\tilde{W}_t)$ is a solution of \eqref{stickyappl}. The statement from
	Lemma \ref{lemmon1} carries over to the limiting processes, since such
	inequalities are preserved under weak convergence, and thus \eqref{monotonappl}
	holds.
	The inequality \eqref{tt2} is implied by Theorem \ref{quantification1} setting $\alpha(t,x)=a(x)=M\ +\ \kappa(x)\cdot x$.
\end{proof}

\begin{proof}[Proof of Lemma \ref{lem:example}]
By \eqref{pimain}, $\pi[(0,\infty)]=\frac{\alpha}{1+\alpha}$ with
 		\begin{eqnarray*}
 			\alpha&:=& \frac{M}{2}\, \int_0^\infty \exp\left(\frac
 			12\int_0^x(M+\kappa(y)\,y)\, dy\right) dx.
 		\end{eqnarray*}
 		In order to provide upper bounds on $\alpha$, we decompose $\alpha=M(a+b)/2$
 		with
 		\begin{eqnarray*}
 			a &=& \int_\mathcal R^\infty \exp\left(\frac
 			12\int_0^x(M+\kappa(y)\,y)\, dy\right)\,dx \qquad\text{and} \\
 			b & = &
 			\int^\mathcal R_0 \exp\left(\frac 12\int_0^x(M+\kappa(y)\,y)\,
 			dy\right)\,dx.
 		\end{eqnarray*}
 		By Condition \eqref{KLRbound}, we have 
 		\begin{eqnarray*}
 			\frac
 			12\int_0^x(M+\kappa(y)\,y)\, dy &=& \frac
 			12\int_0^{\mathcal R}(M+L\,y)\, dy \,+\,\frac
 			12\int_{\mathcal R}^x(M-K\,y)\, dy \\
 			&=& Mx/2-Kx^2/4+(L+K)\mathcal R^2/4
 			\\&=&-K(x-M/K)^2/4+M^2/(4K)+(L+K)\mathcal R^2/4
 		\end{eqnarray*}
 		for $x\geq \mathcal R$ and  
 		\begin{eqnarray*}
 		\frac
 			12\int_0^x(M+\kappa(y)\,y)\, dy &=& \frac
 			12\int_0^{x}(M+L\,y)\, dy\,=\, Mx/2+Lx^2/4
 		\end{eqnarray*}
 		for  $x\leq \mathcal R$. We obtain
 		\begin{eqnarray*}
 			a&=& \exp\left({M^2}/{(4K)}+(L+K)\mathcal R^2/{4}\right) \, 
 			\int_\mathcal R^\infty \exp\left(- K
 			\left(x-{M}/{K}\right)^2/{4}\right)\,dx\\
 			&=& \frac{\sqrt{2}}{\sqrt{K}} \exp\left({M^2}/{(4K)}+(L+K)\mathcal
 			R^2/{4}\right) \int_{(\mathcal R-M/K)\sqrt{K/2}}^\infty
 			\exp\left(-z^2/2\right)\,dz \quad\text{and}\\
 			b&=& \int_0^\mathcal R \exp\left(Mx/2+Lx^2/4\right)\, dx
 		\end{eqnarray*}
 		and give upper bounds for these quantities:
 		\begin{eqnarray}\label{B1}
 			b&\leq& \mathcal R\, \exp\left(M\mathcal R/2+L\mathcal
 			R^2/4\right)
 			\\
 			\label{B2}
 			b &=& \exp\left(M\mathcal R/2+L\mathcal R^2/4\right)\int_0^\mathcal R
 			\exp\left(M(\mathcal R-x)/2-L(\mathcal R^2-x^2)/4\right)\, dx
 			\\\nonumber &=&\exp\left(M\mathcal R/2+L\mathcal R^2/4\right) \int_0^\mathcal R
 			\exp\left(-{M}y/2-{L}y\left(2\mathcal R -y\right)/4\right)\, dy
 			\\ \nonumber &\leq& 
 			\exp\left(M\mathcal R/2+L\mathcal R^2/4\right) 
 			\int_0^\mathcal R
 			\exp\left(-{M}y/2-{L\mathcal R}y/4\right)\, dy 
 			\\ \nonumber &\leq&  \frac{1}{M/2+L\mathcal R/4} \exp\left(M\mathcal
 			R/2+L\mathcal R^2/4\right).
 		\end{eqnarray}
 		Combining \eqref{B1} and \eqref{B2}, we conclude that
 		\begin{eqnarray}\label{B}
 		 b &\leq& \frac{4\mathcal{R}}{\max(4,2M\mathcal{R}+L\mathcal R^2)}
 		 \exp\left(M\mathcal R/2+L\mathcal R^2/4\right).
 		\end{eqnarray}
 		We use the bound $\int_0^\infty e^{-z^2/2} dz \leq \sqrt{2\pi}$ to conclude
 		that
 		\begin{eqnarray}\label{A1}
 			a &\leq& 2 \sqrt{{\pi/K}}  \exp\left({M^2}/{(4K)}+(L+K)\mathcal
 			R^2/{4}\right)
 			\\\nonumber&=& 2 \sqrt{{\pi/K}}   \exp\left(K(R-M/K)^2/4\right)
 		\exp\left(M\mathcal R/2+L\mathcal R^2/4\right) 
 		 			\\\nonumber&\leq & 2 \sqrt{\frac{\pi e}{K} }\exp\left(M\mathcal
 		 			R/2+L\mathcal R^2/4\right)\quad\text{for }K(R-M/K)^2\leq 2.
 		\end{eqnarray}
 		On the other hand, $\int_y^\infty e^{-z^2/2}\, dz \leq e^{-y^2/2}/y$ for any
 		$y>0$ and thus
 		\begin{eqnarray}\label{A2}
 			a &\leq& \frac{2}{K} \frac{1}{\mathcal R-M/K}
 			\exp\left((-K(\mathcal
 			R-M/K)^2+{M^2}/{K}+(L+K)\mathcal
 			R^2)/{4}\right)\\\nonumber
 			&=& \frac{2}{\sqrt{K}} \frac{1}{\sqrt{K(R-M/K)^2}} \exp\left(M\mathcal
 		 			R/2+L\mathcal R^2/4\right)
 			\\&\leq& \frac{\sqrt{2}}{\sqrt{K}} \exp\left(M\mathcal
 		 			R/2+L\mathcal R^2/4\right)\nonumber
 		\end{eqnarray}
 		provided $\mathcal R\geq M/K$ and $K(R-M/K)^2\geq 2$. 
 		Combining \eqref{B}, \eqref{A1} and \eqref{A2}, we obtain in the case
 		$\mathcal{R}\geq M/K$ the bound
 		\begin{eqnarray*}
 		\alpha&=&M(a+b)/2 \\
 		&\leq& \left(\pi^{1/2}e^{1/2}K^{-1/2}+2 \mathcal R
\max (4,L\mathcal R^2+2M\mathcal R)^{-1}\right)\, M\, \exp \left(M\mathcal R/2+L\mathcal R^2/4 \right)
 		\end{eqnarray*}
 		In the case $\mathcal R\leq M/K$, \eqref{B} implies
 		\begin{eqnarray}\label{B3}
 			 b &\leq& \frac{4\mathcal{R}}{\max(4,2M\mathcal{R}+L\mathcal R^2)}
 		 \exp\left(M^2/(4K)+(L+K)\mathcal R^2/4\right).
 		\end{eqnarray}
 		Combining \eqref{B3} and \eqref{A1}, we can conclude for $\mathcal R\leq
 		M/K$ the bound
 		\begin{eqnarray*}
 		\alpha\leq
 		\left(\sqrt{\frac{\pi}{K}}+\frac{2\mathcal{R}}{\max(4,2M\mathcal{R}+L\mathcal
 		R^2)}\right)M \exp\left(\frac{M^2}{4K}+\frac{L+K}{4}\mathcal R^2\right).
 		\end{eqnarray*}
\end{proof}

\begin{proof}[Proof of Theorem \ref{mainMcKeanVlasov}]
 The proof is similar to the proof of Theorem \ref{main1}. We fix $x_0,y_0\in\reals^d$ and corresponding 
 drifts $b(t,x)=b^{x_0}(t,x)$ and $\tilde{b}(t,x)=b^{y_0}(t,x)$ as in \eqref{b1} and \eqref{b2} respectively. Moreover, we choose $\tau_0\in (0,\infty )$
such that \eqref{eqmckeanest} holds for $|\tau |\le \tau_0$. Since
$\vartheta$ is Lipschitz, we can conclude by \eqref{eqmckeanest} that
for any $x\in\reals^d$,
 \begin{eqnarray*}
 |b(t,x)-\tilde b(t,x)|& =& |\tau |\cdot	\norm{\int \vartheta(x,y) \mu_t^{x_0}(dy) - \int \vartheta(x,y) \mu_t^{y_0}(dy) } \\ &\leq&|\tau |L \ \wDist^1(\mu_t^{x_0},\mu_t^{y_0})\ \leq \  
 	L\ A\ e^{-\lambda \, t}\ \norm{x_0-y_0},
 \end{eqnarray*}
where $L$ is the corresponding Lipschitz constant. We can now repeat the 
procedure leading to the proof of Theorem \ref{main1}, replacing $M$ by 
 $\norm{\tau}LAe^{-\lambda \, t}\norm{x_0-y_0}$. In particular, we can conclude
 that there is a coupling $(X_t,Y_t)$ of \eqref{a1} and \eqref{a2} and a
 solution $(r_t,W_t)$ of \eqref{eq:oneDimSticky} with $r_0=|x_0-y_0|$ and drift
 \begin{eqnarray}\label{alphaexplicit}
 \alpha(t,x)=\norm{\tau}L A e^{-\lambda \, t}
 \norm{x_0-y_0} + \kappa(x)  x
 \end{eqnarray}
 such that $\norm{X_t-Y_t}\leq r_t$. Notice that
 Assumption \ref{ass:CurvControlForEta} implies Assumption
 \ref{ass:summaryTimeDep} for the drift $\alpha$.
 We now want to apply Corollary \ref{thmtimedependent}. First, we fix 
 the function $a$ in \eqref{ass:convexityAtInfinity} as $a(\cdot):=\alpha(0,\cdot)$.
 Applying Corollary \ref{thmtimedependent} now yields that for any $0\leq s<
 t$, 
 \begin{eqnarray*}
 		\|{\mu_t^{x_0}-\mu_t^{y_0}}\|_{TV} &\leq & \frac{1}{\epsilon} \
 		\frac{c}{e^{c\,(t-s)}-1} \,\left( e^{-cs} f(\norm{x_0-y_0}) + \int
 		f\,d\pi_0\right) \, + \, \pi_{s}[(0,\infty)].
 \end{eqnarray*}
 By \eqref{eq:pis}, Assumption \ref{ass:summaryTimeDep}, and since 
$f(r)\leq r$, we have $f(|x_0-y_0|)\le |x_0-y_0|$ and $\int f\, d\pi_0<\infty$.
 Moreover, by \eqref{eq:pis}, \eqref{alphaexplicit} and  Assumption \ref{ass:summaryTimeDep},
 \begin{eqnarray*}
 \pi_{s}[(0,\infty)] &\le& \frac 12\,\alpha (s,0)\,\int_0^\infty 
 \exp\left(\frac{1}{2} \int_0^x
 \alpha(s,y)\, dy\right) dx\ \le\ C\, e^{-\lambda s}
 \end{eqnarray*}
 where $C:=\frac 12|\tau | LA\int_0^\infty 
 \exp\left(\frac{1}{2} \int_0^x
 \alpha(s,y)\, dy\right) dx$ \ is a finite constant.
Thus, there is a
 constant $A\in(0,\infty)$
 such that 
 \begin{eqnarray*}
 	\|{\mu_t^{x_0}-\mu_t^{y_0}}\|_{TV} &\leq&  \frac{A}{e^{c\,(t-s)}-1}  \ +
 	C  e^{-\lambda\, s} \, = \, e^{-c\,(t-s)} \frac{A}{1-e^{-c\,(t-s)}} + C
 	e^{-\lambda\, s}
\end{eqnarray*}
for any $0\leq s<t$. We can now set $s=t/2$ and use the
boundedness of $ \|\cdot \|_{TV}$ to see that there is a constant
$B\in(0,\infty)$ such that
 \begin{eqnarray*}
 	\|{\mu_t^{x_0}-\mu_t^{y_0}}\|_{TV} &\leq&  B\, \exp(-\min ({c},\lambda )\, t/2)
 	\qquad\text{for all }t \geq 0.
\end{eqnarray*}
It should be stressed, that the constants $B$ and ${c}$ depend on
the initial conditions.
 \end{proof}

\appendix
\section{Computations for Example \ref{exa:BM}}\label{appendixcalc}

In this appendix we prove lower bounds on the total variation distance between the 
probability measures $\nu (dx)=Z_f^{-1}f(x)\,dx$ and $\mu (dx)=Z_g^{-1}g(x)\,dx$ on $\mathbb R^1$ that have been considered in Example \ref{exa:BM}. Noticing that by symmetry of $f$,
		\begin{eqnarray*}
			Z_g &=& \int_{-\infty}^\infty g(x) \,dx \ =\ \int_{-\infty}^\infty e^{mx}
			f(x)\, dx \ =\ \int_0^\infty (e^{mx}+e^{-mx}) f(x) \, dx \\
			&\geq& 2
			\int_0^\infty f(x)\, dx \ =\ 
			\int_{-\infty}^\infty f(x)\, dx \ =\  Z_f, \qquad\mbox{ we obtain}
		\end{eqnarray*}
			\begin{eqnarray}\nonumber
			\lefteqn{\|\mu -\nu\|_{TV} \ =\ \int_{\mathbb R} \left( 1-{d\mu}/{d\nu}\right)^{+} \, d\nu \, =
			\, \int_{\mathbb R} \left(1- e^{mx}{Z_f}/{Z_g}\right)^{+} \, \nu(dx)}\\
		\nonumber
			 &\geq&  \int_{-\infty}^0 (1-e^{mx}) \,  \nu(dx) 
			\,=\, \int_{0}^\infty 
			\left(1-e^{-mx}\right)\, \nu(dx) \\
			\nonumber
			&=&  {\nu [(0,R)]}\,  \left.\int_0^R (1-e^{-mx})\, dx\right/ R \\ &&+ \, \nu[(R,\infty)]
			\int_R^\infty (1-e^{-mx}) e^{-{k}(x-R)^2/2} \, dx \left/ \int_R^\infty
			e^{-{k}(x-R)^2/2}\, dx \right.\label{appe2}\\\nonumber
			&=&\nu [(0,R)]\left(mR-1+e^{-mR}\right)/{(mR)}\\
			&&+\, \nu [(R,\infty)] \int_0^\infty (1-e^{-m(R+t)})
			e^{-{k}t^2/2}\, dt \left / \int_0^\infty e^{-{k}t^2/2}\, dt\right. .\nonumber
		\end{eqnarray}
		Using that $(e^{-x}-1+x)/x \leq 1-e^{-x}$ for any $x>0$, we obtain the lower
		bound
		\begin{eqnarray*}
			\|\mu -\nu\|_{TV} &\geq& {(e^{-mR}-1+mR)}/{(mR)}.
		\end{eqnarray*}
		We now derive an improved bound for small $k$. Suppose that $R\sqrt{k}\leq 1$. Then
		\begin{eqnarray*}
			{\nu [(R,\infty)]}/{\nu [(0,R)]}&=& \left.\int_0^\infty
			e^{-{k}t^2/2}\, dt\right/ R \ =\  \sqrt{\pi/(2k)}\, {R^{-1}} 
		\end{eqnarray*}
		implies
		$
			\nu [(R,\infty)] 
			\,=\, \frac{1}{2} \left(1\,+\, {R\sqrt{2 k/\pi}}\right)^{-1}\, \geq\, 
			\frac{1}{4}$.
		Hence by \eqref{appe2},
		\begin{eqnarray*}
			\|\mu -\nu\|_{TV} &\geq& \frac{1}{4} \int_0^\infty (1-e^{-m(R+t)}) e^{-{k}t^2/2}\, dt
			\left / \int_0^\infty e^{-{k}t^2/2}\, dt\right.
			\\&=& \frac{1}{4}
			\left(1-e^{-mR+m^2/(2k)}\left(1-\sqrt{{2}/{\pi}}\int_0^{m/\sqrt{k}}
			e^{-s^2/2}\, ds\right)\right)
			\\&\geq &\frac{1}{4} \left(1-e^{-mR+m^2/(2k)}+\sqrt{{2}/{(\pi k)}}{m}e^{-mR}\right).
		\end{eqnarray*}

 \bibliographystyle{plain}
 \bibliography{bib}
 
 \end{document}